\documentclass[journal]{IEEEtran}
\usepackage[short]{optidef}
\usepackage{graphicx, float}
\usepackage[subtle]{savetrees}
\usepackage{amsmath}
\usepackage{siunitx}
\usepackage{amssymb}
\usepackage{amsthm}
\usepackage{booktabs}
\usepackage{cite}
\usepackage{enumerate}
\usepackage{changes}
\definechangesauthor[name=Nawaf, color=blue]{1}
\ifCLASSOPTIONcompsoc
    \usepackage[caption=false, font=normalsize, labelfont=sf, textfont=sf]{subfig}
\else
\usepackage[caption=false, font=footnotesize]{subfig}
\fi
\newtheorem{theorem}{Theorem}[section]

\newtheorem*{remark}{Remark}

\IEEEoverridecommandlockouts  

\begin{document}

\title{Convex inner approximation of the feeder hosting capacity limits on dispatchable demand}

\author{Nawaf Nazir, \textit{Student Memeber, IEEE} and Mads Almassalkhi, \textit{Member, IEEE} 
\thanks{N. Nazir and M. Almassalkhi are with the Department of Electrical and Biomedical Engineering, University of Vermont, Burlington, Vermont, USA {\tt\small \{mnazir,malmassa\}@uvm.edu} Support from U.S. Department of Energy award number DE-EE0008006 is gratefully acknowledged.}}

 \maketitle

\begin{abstract}
This paper presents a method to obtain a convex inner approximation that aims to improve the feasibility of optimal power flow (OPF) models in distribution feeders. For a resistive distribution network, both real and reactive power effect the node voltages and this makes it necessary to consider both when formulating the OPF problem. Inaccuracy in linearized OPF models may lead to under and over voltages when dispatching flexible demand, at scale, in response to whole-sale market or grid conditions. In order to guarantee feasibility, this paper obtains an inner convex set in which the dispatchable resources can operate, based on their real and reactive power capabilities, that guarantees network voltages to be feasible. Test simulations are conducted on a standard IEEE distribution test network to validate the approach.
\end{abstract}

\IEEEpeerreviewmaketitle

\section{Introduction}\label{sec:introduction}
With the increasing penetration of renewable generation and demand side flexibility in distribution networks, network constraints such as voltage limits could be violated. Traditional optimization techniques for dispatching resources include linear OPF based on \textit{LinDist} models~\cite{baran1989optimal}. However, these linear models only work well close to the operating point and as the system is stressed to its extremes due to increasing penetration of DERs, they break down. Recently, improved linear approximate models have been developed that provide better accuracy over wide range of operating points~\cite{yang2016optimal,bolognani2016fast,bolognani2016existence}. However, the solution space of the power flow equations is highly non-convex~\cite{hiskens2001exploring}, which means that such methods cannot guarantee the feasibility of its solutions under all conditions. A comprehensive review of the many linear load flow formulations can be found in~\cite{stott2009dc}.

Apart from linear OPF methods, convex relaxation based techniques have gained popularity in recent years. Convex relaxation based methods such second order cone programs (SOCP) and semi-definite programs (SDP) provide lower bounds on the optimal solution~\cite{lavaei2012zero,farivar2011inverter}. However, these convex relaxations are not always exact~\cite{lesieutre2011examining}. The conditions under which the convex relaxations are exact have been studied extensively in~\cite{lavaei2012zero,lavaei2014geometry,gan2015exact} and often these conditions break down under extreme penetration of renewables and reverse power flow~\cite{huang2017sufficient}. Furthermore, the exactness of the relaxations is dependent on the chosen objective function~\cite{li2016convex}.

In order to improve the feasibility of OPF solutions in distribution systems with extreme penetration of renewables, this papers aims to develop a method that guarantees the feasibility of optimized solutions. In order to achieve this, a convex inner approximation is developed to determine the feasible operating region for the dispatchable resources in the distribution network. Previous works in literature such as~\cite{dvijotham2016error,heidari2017non} have developed techniques to determine error bounds in linear power flow approximations. This paper builds upon these works but develops a convex inner approximation for determining the operating region of dispachable resources that guarantees feasibility of solution.

As the proportion of dispatchable demand-side resources increases in the distribution network, they are expected to provide flexibility to the grid in the form of valuable energy services~\cite{kristov2016tale}. These flexible resources could be a fleet of DERs that constitute a virtual battery (VB), solar PV arrays, or advanced distribution feeders schemes act as a VB resource. In either case, these resources in aggregate are expected to provide certain energy services and participate in ISO markets, such as real-time or ancillary market services. However, the resulting ISO market-based dispatch signal does not consider the underlying distribution network and nodal constraints. Disaggregating the market-based dispatch signal at a nodal resource level, in real time, to account for local constraints and grid conditions represents a  challenging problem. The key contribution of this paper is a convex formulation that provides an aggregator with the ability to disaggregate a fleet-wide dispatch signal into a feasible nodal dispatch across a distribution network as depicted in Fig.~\ref{fig:cyber_physical_model}. Specifically, this paper develops a technique to determine a feasible operating region of these dispatchable resources which does not violate local network constraints. This is achieved by developing a provable convex inner approximation of the feasible region. Simulation tests are conducted on IEEE-13 node system~\cite{kersting2001radial} to show the effectiveness and validity of the approach.
\begin{figure}[t]
\centering
\includegraphics[width=0.4\textwidth]{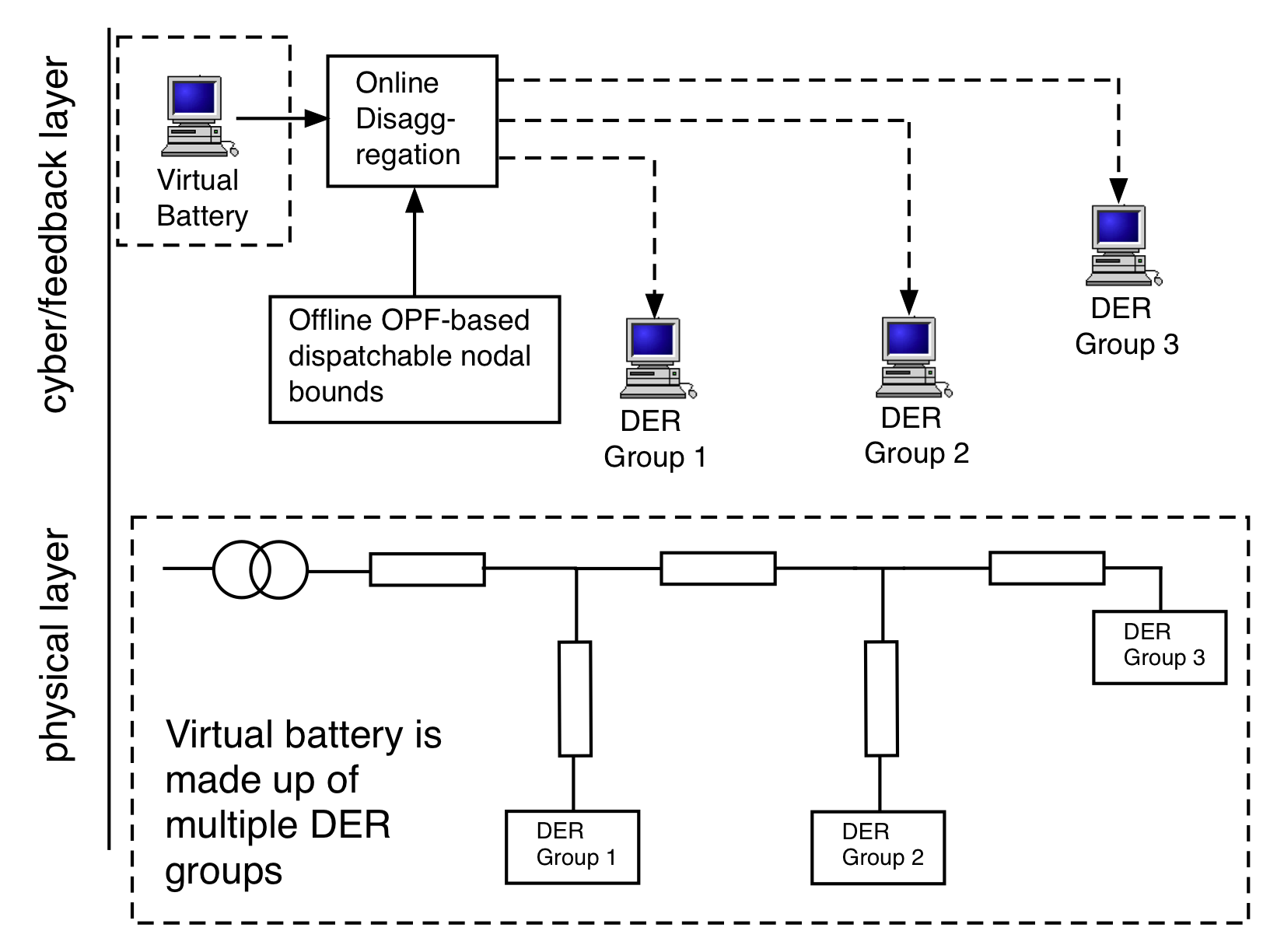}
\caption{\label{fig:cyber_physical_model}A schematic representation of the network model. The physical layer represents the circuit that connects the different DER groups into an aggregate virtual battery, whereas the cyber layer represents the disaggregation of the virtual battery market signal to the DER groups based on the feasible nodal bounds that are determined offline.}
\end{figure}
The main contributions of the paper are follows:
\begin{enumerate}
    \item Through a motivating example, this paper explains the shortcomings of linear OPF approximations and how they can violate network constraints.
    \item The problem of determining the feasible operating region of dispactchable resources is re-formulated as an inner convex optimization that respects network constraints.
\end{enumerate}
The rest of the paper is organized as follows: Section~\ref{sec:DC_OPF} illustrates the shortcomings of linear OPF under high renewable penetration in distribution networks. Section \ref{sec:Opt_form} develops the mathematical model for the optimization problem of determining the operating region of dispatchable resources, whereas section \ref{sec:convex_inner} provides the convex inner approximation of the feasible space. Simulation results showing the validity of the approach are given in section \ref{sec:sim_results} and finally conclusions and future scope of work are provided in section \ref{sec:conclusion}.

\section{Shortcomings of linear OPF}\label{sec:DC_OPF}
This section presents the shortcomings of linear OPF approximations under certain conditions in distribution networks. Simulations are run on the modified IEEE-13 node test case to check the effect of real and reactive power variation on nodal voltages. The IEEE-13 node test case with a DER at node $6$ capable of four-quadrant operation is shown in Fig.~\ref{fig:13nodes}. For the purpose of this study, the switching devices in the network (switches, capacitor banks, transformers) are assumed to be fixed at their nominal values. The real and reactive power injections of the DER are varied independently to observe the effect on node voltages. The results are shown in Fig.~\ref{fig:phaseA_Q} and Fig.~\ref{fig:phaseA_P}. From the figures it can be observed that changes in real and reactive power injections at one node have significant effect on voltages at other nodes, especially the nodes which are "down-hill" from the injection node. Unlike transmission systems, where the coupling between real power and voltage is minimal, in distribution systems, changes in real power injection can cause significant( if not as much as reactive power) change in node voltages. From these results it becomes clear that the effect of both real and reactive power needs to be considered in order to correctly control voltages in distribution systems. However, in most linear power flow approximations, the affect of real power variations on the nodal voltages is often neglected, which could result in violation of voltage constraints. 

\begin{figure}[ht]
\centering
\includegraphics[width=0.37\textwidth]{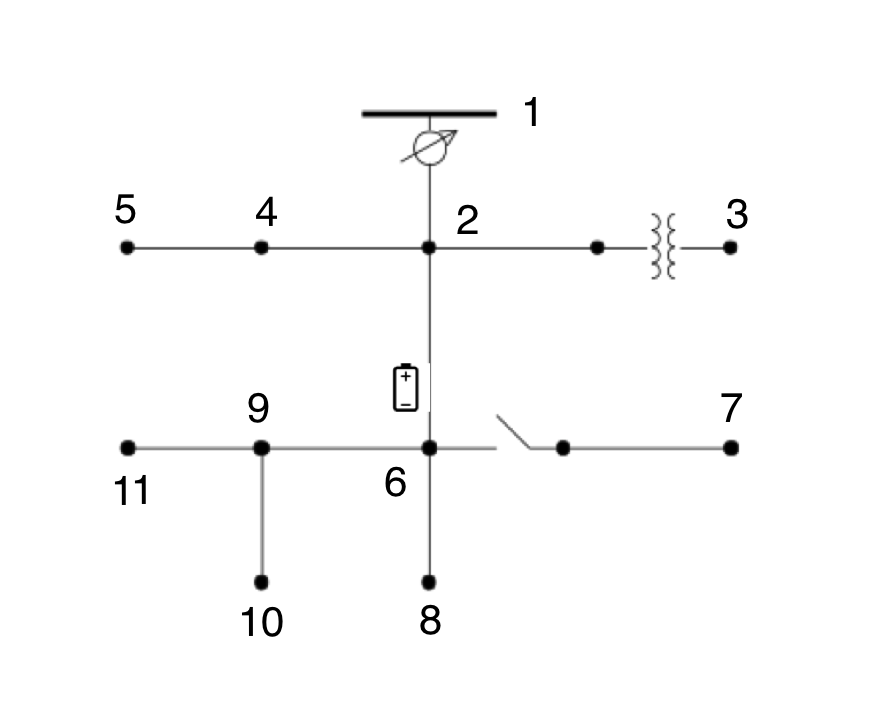}
\caption{\label{fig:13nodes}Modified IEEE-13 node test system with DER capable of four-quadrant operation at node $6$. The switch between nodes $6$ and $7$ is assumed to be closed and the transformer between nodes $2$ and $3$ is assumed to be ideal with unity turns ratio}
\end{figure}
 
\begin{figure}
    \centering
    \subfloat[\label{fig:phaseA_Q}]{
    \includegraphics[width=0.52\linewidth]{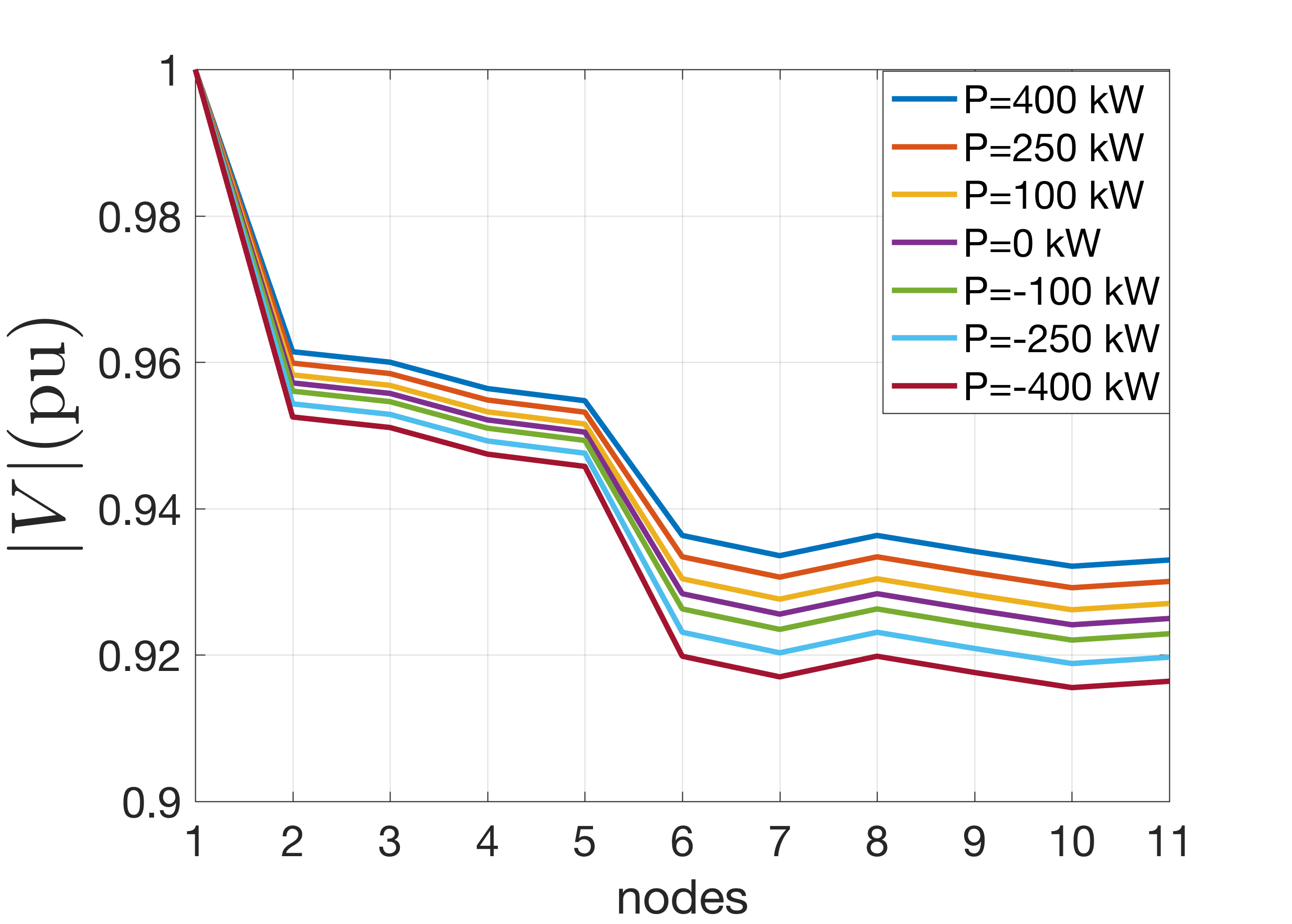}}
    \subfloat[\label{fig:phaseA_P}]{
    \includegraphics[width=0.52\linewidth]{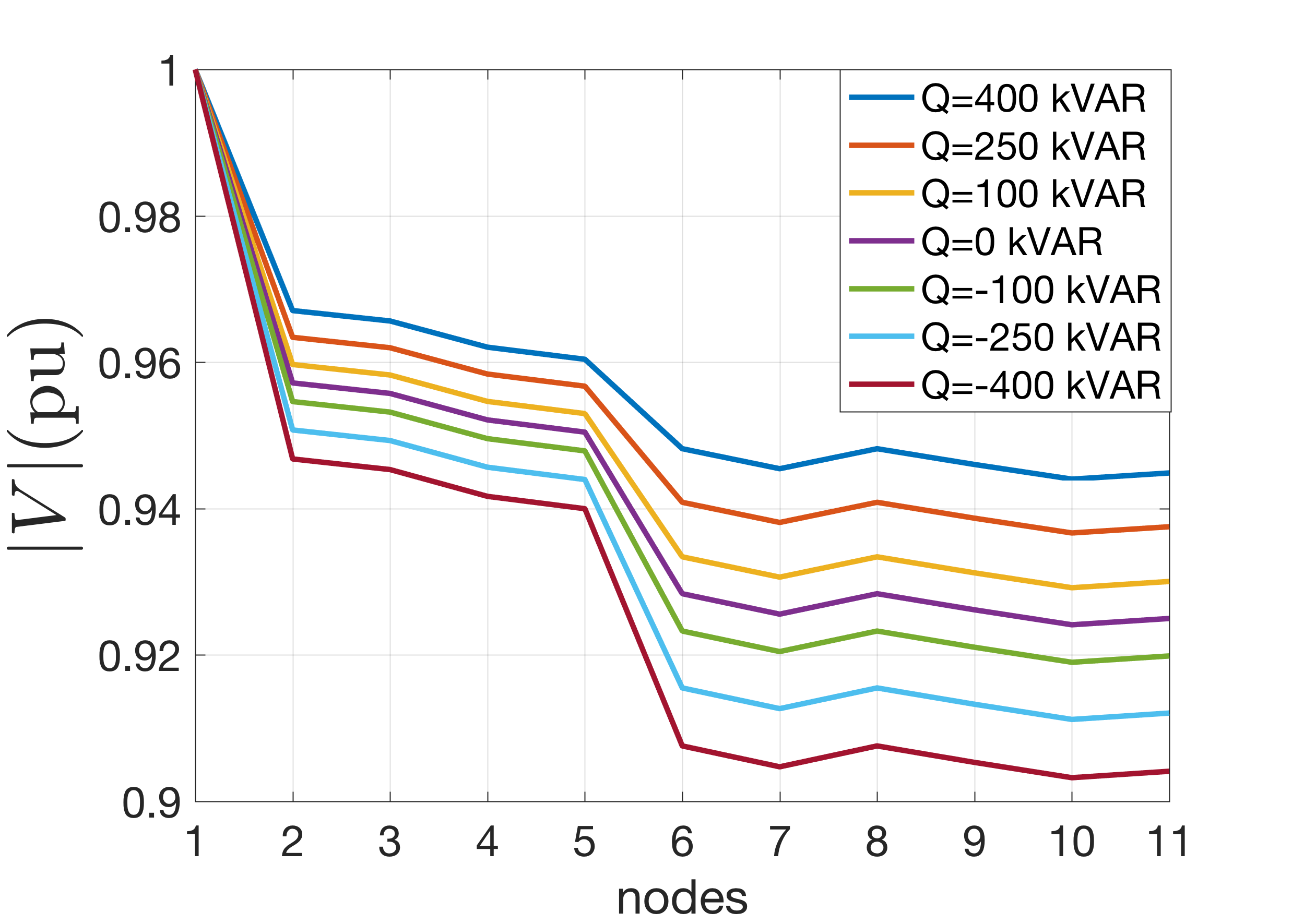}}
\caption{(a) Variation of node voltage with change in real power injection at node $6$, (b) Variation of node voltage with change in reactive power injection at node $6$. }
\end{figure}

\begin{figure}[ht]
\centering
\includegraphics[width=0.37\textwidth]{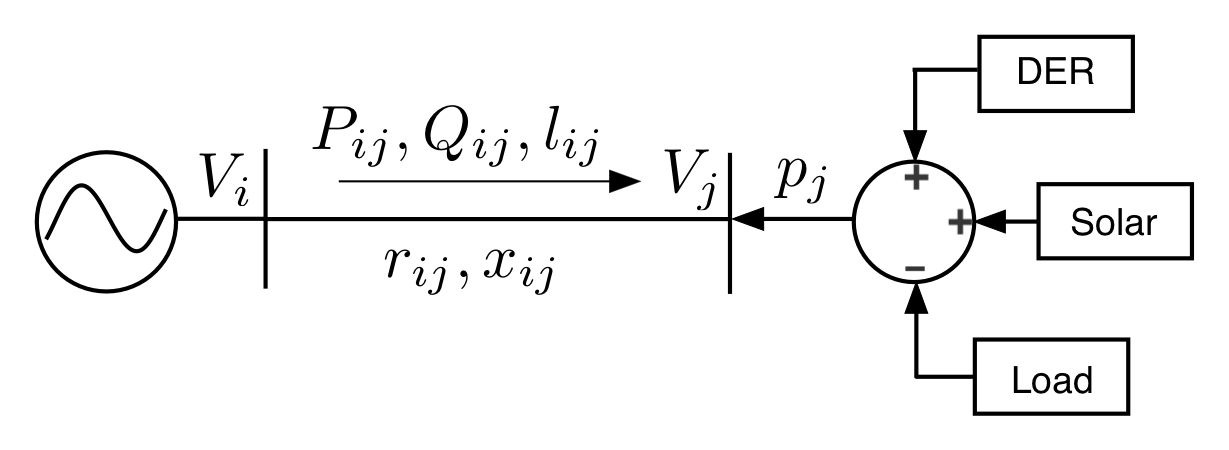}
\caption{\label{fig:two_node}Two node model with demand side power source.}
\end{figure}


\begin{figure}
    \centering
    \subfloat[\label{fig:QvsV}]{
    \includegraphics[width=0.52\linewidth]{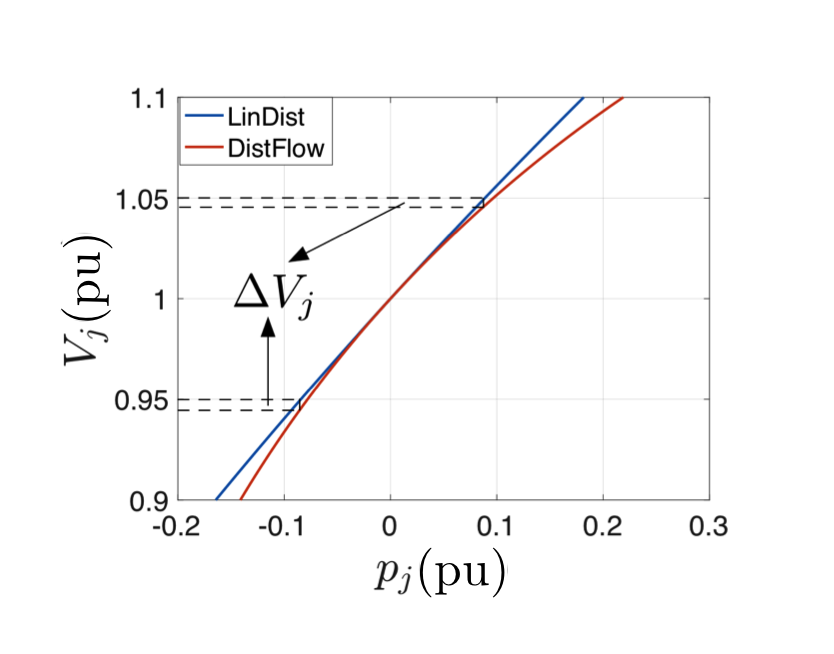}}
    \subfloat[\label{fig:feas_region}]{
    \includegraphics[width=0.52\linewidth]{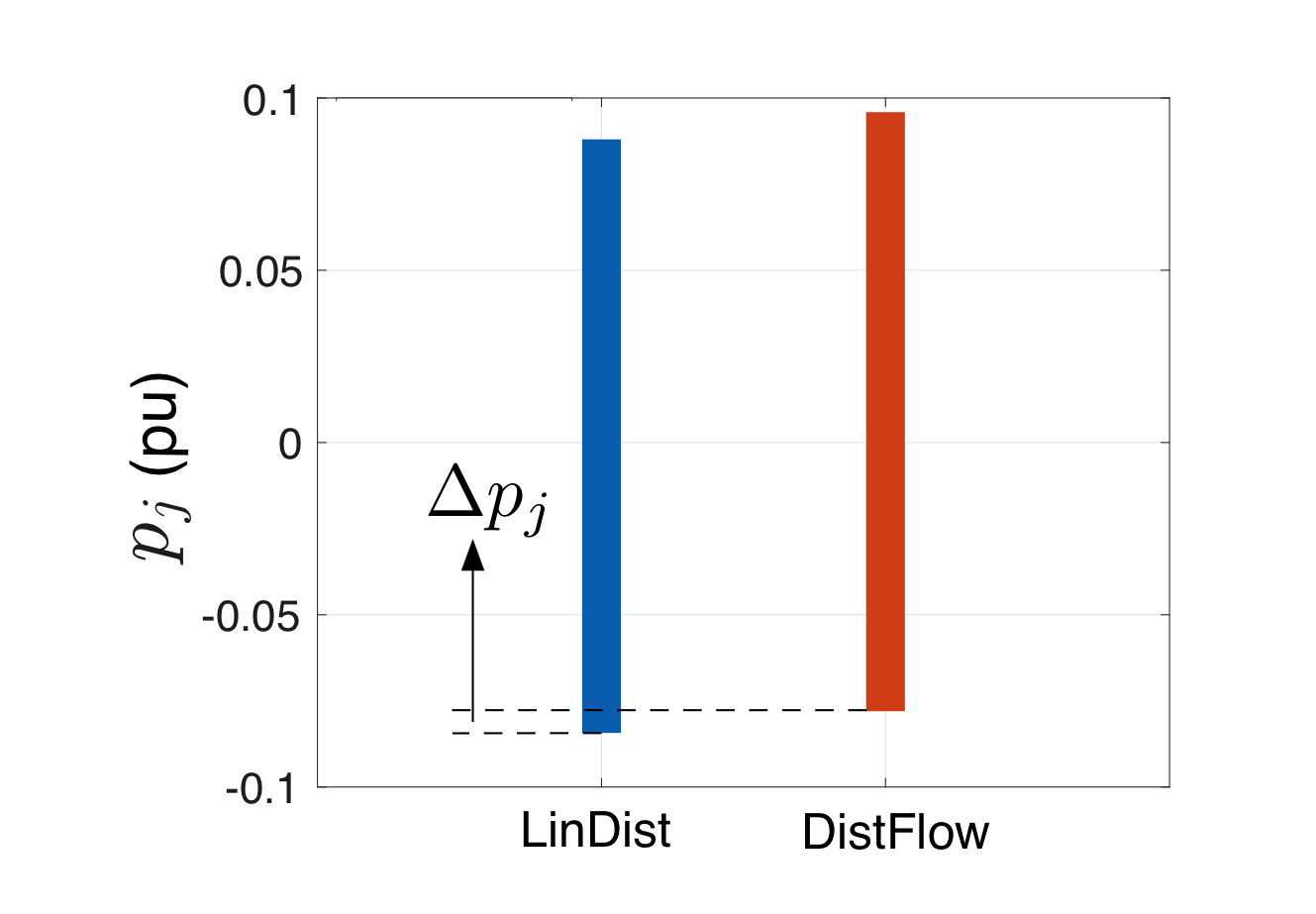}}
\caption{(a) Comparison in change in node voltage with change in real power set-points between \textit{DistFlow} and \textit{LinDist} for two node system. The two models result in different voltages and if design is based on linear model then network voltages limit will be violated if $p_j$ is operated at its lower limit, (b) Comparison of feasible region obtained from \textit{LinDist} and \textit{DistFlow} models for two node model. }
\end{figure}


To further illustrate this point, a two node model with DERs, solar and demand, as shown in Fig.~\ref{fig:two_node} and parameters given in Table~\ref{table_twonode} is considered. For this system, the net power injection at node $j$, $p_j$, which represents the net effect of DERs, solar and demand, is varied over a range to find the corresponding voltages obtained from \textit{LinDist} and \textit{DistFlow}. From Fig.~\ref{fig:QvsV} it can be seen that the linear model and nonlinear \textit{DistFlow} model do not match and this could lead to operating the system at set-points that violate the voltage constraints as can be seen from Fig.~\ref{fig:QvsV}, where the voltage violation occurs when the system is operated at the lower voltage limit. This implies that when formulating an OPF model, it is important to consider the effect of the non-linear terms in the power flow equations in order to ensure feasibility. 

The proposed approach in this paper, takes a worst case of the non-linear terms in the power flow equations and develops a feasible model for determining the real power limit bounds on the net power injections, which in turn can be used to determine the bounds on the flexible resources assuming demand and solar forecast are known. This means that the technique guarantees node voltages to be within their limits for the determined power bounds, while at the same time keeping the convex form of the formulation.
\begin{table}[h!]
\centering
\caption{\label{table_twonode}Parameters of the 2-node, 1-line system}
{
\begin{tabular}{rcl} 
\toprule
{Symbol} & {Type} & {Bounds/Values} \\
\midrule
            $l_{ij}$ & Variable & [0,0.5] p.u\\
            $V_j$ & Variable & [0.95,1.05] p.u\\
            $P_{ij}$ & Variable & [-5,5] MW\\
            $Q_{ij}$ & Variable & [-5,5] MVar\\
            $p_j$ & Variable & [-1,1] MW\\
            $V_{i}$ & Data & 1 p.u\\
            $r_{ij}+jx_{ij}$ & Data & 10+j15 ohm\\
            $V_{\text{base}}$ & Data & 4.16 kV\\
            $S_{\text{base}}$ & Data & 1 MVA
            \\
            \bottomrule
\end{tabular}
}
\end{table}

\section{Formulation of the optimization problem}\label{sec:Opt_form}
The aim of this paper is to formulate an optimization scheme to provide feasible operating limits to the dispatchable resources based on the real and reactive power capabilities of the network, in order to satisfy the nodal voltage constraints. The input to this allocation problem is the real and reactive power capabilities of each node and the output is the real power operating region where each node can operate while satisfying the network constraints. In order to increase the operating region of the resources, the optimization scheme is posed as a power bound maximization problem. \textit{DistFlow} equations as given in~\cite{baran1989optimal} are used to solve the optimization problem. 
However, \textit{DistFlow} equations are non-linear which takes the problem outside the realm of convex optimization. Linearized \textit{LinDist} models are often used, but they are accurate only close to the operating point voltages. As the system starts operating away from the nominal voltage, the errors could become large~\cite{alyami2014adaptive}. In this work we develop a feasible convex formulation by modifying the \textit{LinDist} equations. However, the techniques presented here can easily be extended to other single phase and multi-phase linearized power flow models in literature such as~\cite{yang2018linearized,bernstein2017linear}. 
\subsection{Mathematical model}\label{sec:math_model}
Consider a radial distribution network as a graph $\mathcal{G}=\{\mathcal{N}\cup\{0\},\mathcal{L}\}$ consisting of $\mathcal{N}:=\{1,\hdots,|\mathcal{N}|\}$ nodes and a set of $\mathcal{L}:=\{1,\hdots,|\mathcal{L}|\}$ branches such as the one shown in Fig.~\ref{fig:radial_network}. Node $0$ is assumed to be the substation node with a fixed voltage $V_0$. Let $B\in \mathbb{R}^{(n+1)\times n}$ be the \textit{incidence matrix} of the undirected graph $\mathcal{G}$ relating the branches in $\mathcal{L}$ to the nodes in $\mathcal{N}\cup \{0\}$, such that the entry at $(i,j)$ of $B$ is $1$ if the $i$-th node is connected to the $j$-th branch and otherwise $0$. If $V_i$ and $V_j$ are the voltage phasors at nodes $i$ and $j$ and $I_{ij}$ is the current phasor in branch $(i,j)\in \mathcal{L}$, then $v_i=|V_i|^2$, $v_j=|V_j|^2$ and $l_{ij}=|I_{ij}|^2$. $P_{i}$ be the real power flow, $Q_{i}$ be the reactive power flow, $p_i$ be the real power injection and $q_i$ be the reactive power injection, at bus $i$, $r_{ij}$ and $x_{ij}$ be the resistance and reactance of the branch $(i,j)\in \mathcal{L}$ and $z_{ij}=r_{ij}+jx_{ij}$ be the impedance. Then for the radial distribution network, the relation between node voltages and power flows is given by the following \textit{DistFlow} equations:
\begin{align}
v_j=&v_i+2r_{ij}P_{j}+2x_{ij}Q_{j}-|z_{ij}|^2l_{ij} \quad \forall (i,j)\in \mathcal{L}\label{eq:volt_rel}\\
P_{i}=&P_j+p_i-r_{ij}l_{ij} \quad \forall (i,j)\in \mathcal{L}\label{eq:real_power_rel}\\
Q_i=&Q_j+q_i-x_{ij}l_{ij} \quad \forall (i,j)\in \mathcal{L}\label{eq:reac_power_rel}\\
l_{ij}=&\frac{P_j^2+Q_j^2}{v_j} \quad \forall (i,j)\in \mathcal{L}\label{eq:curr_rel}
\end{align}

\begin{figure}[ht]
\centering
\includegraphics[width=0.43\textwidth]{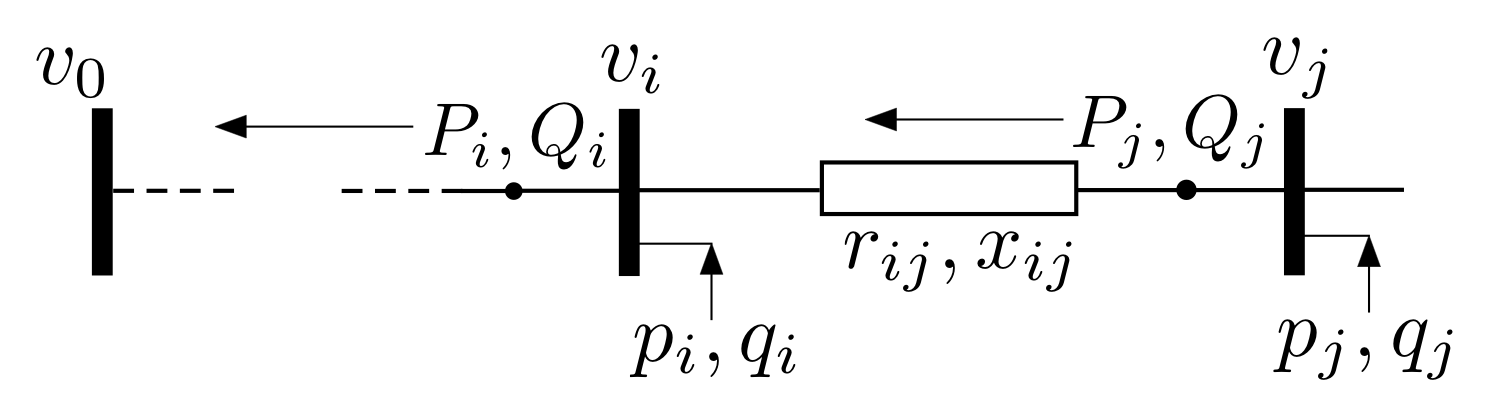}
\caption{\label{fig:radial_network} Diagram of a radial distribution network~\cite{heidari2017non}.}
\end{figure}

The above set of equations are nonlinear due to \eqref{eq:curr_rel}. By leveraging the \textit{incidence} matrix $B$ of the distribution network and following the method adopted in~\cite{heidari2017non}, \eqref{eq:real_power_rel} and \eqref{eq:reac_power_rel} can be expressed through the following matrix equations:
\begin{align}
    P=&p+AP-ARl\label{eq:P_matrix}\\
    Q=&q+AQ-AXl\label{eq:Q_matrix}
\end{align}
where $P=[P_i]_{i\in \mathcal{N}}$, $Q=[Q_i]_{i\in \mathcal{N}}$, $p=[p_i]_{i\in \mathcal{N}}$, $q=[q_i]_{i \in \mathcal{N}}$, $R=\text{diag}\{r_{ij}\}_{(i,j)\in \mathcal{L}}$, $X=\text{diag}\{x_{ij}\}_{(i,j)\in \mathcal{L}}$, $l=[l_{ij}]_{(i,j)\in \mathcal{L}}$ and $A=[0_n \quad I_n]B-I_n$, where $I_n$ is the $n\times n$ identity matrix and $0_n$ is a column vector of $n$ rows.
Simplifying \eqref{eq:P_matrix} and \eqref{eq:Q_matrix} leads to the following expression for $P$ and $Q$
\begin{align}
    P=&Cp-D_{\text{R}}l\label{eq:P_relation}\\
    Q=&Cq-D_{\text{X}}l\label{eq:Q_relation}
\end{align}
where $C=(I_n-A)^{-1}$, $D_{\text{R}}=(I_n-A)^{-1}AR$, and $D_{\text{X}}=(I_n-A)^{-1}AX$. 

\begin{remark}
The matrix $(I_n-A)$ is nonsingular since $I_n-A=2I_n-[0_n \quad I_n]B=2I_n-B_n$, where $B_n:=[0_n \quad I_n]B$ is the $n\times n$ matrix obtained by removing the first row of $B$. For a radial network, the vertices and edges can always be ordered in such a way that $B$ and $B_n$ are upper triangular with $\text{diag}(B_n) = 1_n$, which implies that .  $2I_n-B_n$ is upper triangular and $\text{diag}(2I_n-B_n)=1_n$. Thus,  the $\det(2I_n-B^{'})=1>0$ and $I_n-A$ is then non-singular.
\end{remark}

Similarly, \eqref{eq:volt_rel} can be applied recursively to the distribution network in Fig.~\ref{fig:radial_network} to get the following matrix equation:
\begin{align}\label{eq:volt_matrix_rel}
    [v_j-v_i]_{(i,j)\in \mathcal{L}}=2(RP+XQ)-Z^2l
\end{align}
where $Z^2=\text{diag}\{z_{ij}^2\}_{(i,j)\in \mathcal{L}}$. Based on the \textit{incidence matrix} $B$, the left hand side of \eqref{eq:volt_matrix_rel} can be formulated in terms of the fixed head node voltage as:
\begin{align}\label{eq:volt_transform}
    C^T[v_j-v_i]_{(i,j)\in \mathcal{L}}=V-v_{\text{0}} \mathbf{1}_n
\end{align}
where $V=[v_i]_{i\in \mathcal{N}}$. Based on \eqref{eq:volt_transform}, \eqref{eq:volt_matrix_rel} can be expressed as:
\begin{align}\label{eq:volt_matrix_2}
    V=v_{\text{0}} \mathbf{1}_n+2(C^TRP+C^TXQ)-C^TZ^2l
\end{align}
Substituting \eqref{eq:P_relation} and \eqref{eq:Q_relation} into \eqref{eq:volt_matrix_2}, we obtain a compact relation between voltage and power injections shown below.
\begin{align}\label{eq:final_volt_rel}
    V=v_{\text{0}}\mathbf{1}_n+M_{\text{p}}p+M_{\text{q}}q-Hl
\end{align}
where $M_{\text{p}}=2C^TRC$, \quad $M_{\text{q}}=2C^TXC$ and \newline $H=C^T(2(RD_{\text{R}}+XD_{\text{X}})+Z^2)$

Apart from the nonlinear relation \eqref{eq:curr_rel} of $l$ to $P,Q$ and $V$, \eqref{eq:final_volt_rel} is a linear relationship between the nodal power injections $p,q$ and node voltages $V$. The nonlinearity in the network is represented by~\eqref{eq:curr_rel}, as the current term $l$ is related to the power injections and node voltages in a nonlinear fashion. Including this term into the optimization model would render the optimization problem NP hard, however, neglecting this term could result in infeasible solutions from the linearized OPF model. 


\subsection{Optimization problem formulation}
The problem being addressed in this paper is to determine the convex feasible operating region of dispatchable resources that respects the network voltage constraints. 
Let $\Delta p:=p^+ - p^-$ be the feasible operating region of the net power injections. If the feasible region of the net power injections is found, then the feasible operating region of the flexible resources can be easily determined assuming the demand and solar power forecast are known. Based on these assumptions, \eqref{eq:final_volt_rel} can be applied at $p^+$ and $p^-$ as:
\begin{align}
    V^+=&v_{\text{0}} \mathbf{1}+M_{\text{p}}p^++M_{\text{q}}q^+-Hl^+\label{eq:volt_rel_pos}\\
    V^-=&v_{\text{0}} \mathbf{1}+M_{\text{p}}p^-+M_{\text{q}}q^--Hl^-\label{eq:volt_rel_neg}
\end{align}

where $V^+:=V(p^+)$, $q^+:=q(p^+)$, and $l^+:=l(p^+)$ are the respective variable values at $p^+$, and, $V^-:=V(p^-), q^-:=q(p^-), l^-:=l(p^-)$ are the values of the variables at $p^-$.
Next, we modify \eqref{eq:volt_rel_pos} and \eqref{eq:volt_rel_neg} to obtain the feasible operating region as:
\begin{align}
     M_{\text{p}}p^+=&V^+-v_0\mathbf{1}-M_{\text{q}}q^+ +Hl^+\label{eq:op_reg_rel_pos}\\
     M_{\text{p}}p^-=&V^--v_0\mathbf{1}-M_{\text{q}}q^- +Hl^-\label{eq:op_reg_rel_neg}
\end{align}


Based on \eqref{eq:op_reg_rel_pos} and \eqref{eq:op_reg_rel_neg}, the optimization problem to determine the maximum feasible operating region of the network can be obtained from the solution of the two optimization problems shown below for $p^+$ and $p^-$:
\begin{subequations}\label{eq:P1}
\begin{align}\label{eq:P1_a}
\text{(P1)} \quad &  \max_{V^+, p^+, q^+, l^+} \ \sum_{i=1}^{n}\log( p_i^+) & \\
\text{subject to}: \quad & \eqref{eq:op_reg_rel_pos}, \eqref{eq:curr_rel} \label{eq:P1_b}\\
  & \underline{S_i} \le f(p_i^+,q_i^+) \le \overline{S_i}  \forall i\in \mathcal{N}\label{eq:P1_c}\\
   & \underline{V}\leq  V^+ \leq \overline{V}\label{eq:P1_d}\\
    &  \underline{l}\leq  l^+ \leq \overline{l}\label{eq:P1_e}
\end{align}
\end{subequations}
where $f(p_i^+, q_i^+)$ represents the type of apparent power constraint on the nodal injections that is required to satisfy the bounds $\underline{S_i}\in \mathcal{R}$ and $\overline{S_i} \in \mathbb{R}$ at node $i$. $f(p_i^+,q_i^+)$ could represent box constraints  on active and reactive power or it could represent a quadratic apparent power constraint in the case of an inverter. $\underline{V} \in \mathbb{R}^n$ and $\overline{V}\in \mathbb{R}^n$ are the voltage magnitude square lower and upper limits, $\underline{l} \in \mathbb{R}^n$ and $\overline{l} \in \mathbb{R}^n$ are the current magnitude square lower and upper limits.
The optimization problem \eqref{eq:P1} determines the maximum power that can be supplied by the dispatchable resources in the distribution network. The optimization problem to find the minimum power that can be supplied can similarly be determined based on \textit{(P2)}.
\begin{subequations}\label{eq:P2}
\begin{align}\label{eq:P2_a}
\text{(P2)} \quad & \max_{V^-, p^-, q^-, l^-} \ \sum_{i=1}^{n}\log(- p_i^-)\\
\text{subject to}: \quad &  \eqref{eq:op_reg_rel_neg}, \eqref{eq:curr_rel} \label{eq:P2_b}\\
  & \underline{S_i}\le  f(p_i^-,q_i^-) \le \overline{S_i} \quad \forall i\in \mathcal{N}\label{eq:P2_c}\\
    & \underline{V}\leq  V^- \leq \overline{V}\label{eq:P2_d}\\
    & \underline{l}\leq  l^- \leq \overline{l}\label{eq:P2_e}
\end{align}
\end{subequations}
It is assumed that $p_i^- < 0, \forall i\in \mathcal{N}$.
Optimization problems \textit{(P1)} and \textit{(P2)} are non-convex due to the constraint \eqref{eq:curr_rel} which is a nonlinear relationship. In the next section we will provide a formulation for these problems that is a linear inner approximation.

\section{Convex inner approximation formulation}\label{sec:convex_inner}
In order to obtain a convex inner approximation of \textit{(P1)} and \textit{(P2)}, we need to approximate the nonlinear relationship in \eqref{eq:curr_rel}. This is obtained by considering the worst case of $l$ in order to obtain a conservative estimate of $\Delta p$ (i.e., $l_{\text{min}}$ in \textit{(P1)} and $l_{\text{max}}$ in \textit{(P2)}). Based on this approximation, a conservative estimate of $\Delta p$, $\Delta p_{\text{c}}=p^+_{\text{c}}-p^-_{\text{c}}$ can be obtained, where $p^+_{\text{c}}, p^-_{\text{c}}$ are the inner approximations of $p^+,p^-$.
\begin{align}
     M_{\text{p}}p^+\ge &V^+-v_{\text{0}}\mathbf{1}-M_{\text{q}}q^+ +Hl_{\text{min}}\label{eq:op_cons_rel_pos}=M_{\text{p}}p^+_c\\
     M_{\text{p}}p^-\le&V^--v_{\text{0}}\mathbf{1}-M_{\text{q}}q^- +Hl_{\text{max}}=M_{\text{p}}p^-_c\label{eq:op_cons_rel_neg}
\end{align}
Based on the above inner approximations, \textit{(P1)} and \textit{(P2)} can now be modified to a convex inner approximation model shown below, if $l_{\text{min}},l_{\text{max}}$ can be found \textit{a-priori} based on the network capacity.
\begin{subequations}\label{eq:P3}
\begin{equation}
\text{(P3)} \max_{V^+, p^+_{\text{c}}, q^+} \ \sum_{i=1}^{n}\log(p_{i,\text{c}}^+)
\end{equation}
\begin{equation}
\text{subject to}:\quad  \eqref{eq:op_cons_rel_pos}, \eqref{eq:P1_c}-\eqref{eq:P1_d}
\end{equation}
\end{subequations}

\begin{subequations}\label{eq:P4}
\begin{equation}
\text{(P4)} \max_{V^-, p^-_{\text{c}}, q^-} \ \sum_{i=1}^{n}\log(-p_{i,\text{c}}^-)
\end{equation}
\begin{equation}
\text{subject to}: \quad \eqref{eq:op_cons_rel_neg}, \eqref{eq:P2_c}-\eqref{eq:P2_d}
\end{equation}
\end{subequations}
The optimization problems \textit{(P3)} and \textit{(P4)} are convex and determine an inner approximation of the feasible operating region of the nodal injections that satisfy the network constraints. The problem then is to determine the the worst cases of $l$, i.e., $l_{\text{min}}$ and $l_{\text{max}}$, which is discussed in the next section.

\subsection{Determining the worst-case bounds for $l$}
In this section we present a method to calculate the worst case of $l$, i.e., $l_{\text{min}}$ and $l_{\text{max}}$, which results in a feasible convex inner approximation. $l_{\text{min}}$ and $l_{\text{max}}$ are calculated based on the real and reactive power capacity of the nodal injections, which in turn is determined based on the demand and solar profile and DER capacities at a particular node. Before solving the optimization problems \textit{P3} and \textit{P4}, we determine $l_{\text{min}}$ and $l_{\text{max}}$ by solving a power flow based on the real and reactive power capacity of the nodal injections as depicted by the block diagram in Fig.~\ref{fig:block_model} that outlines the steps to solve the optimization problem. For simplicity, in this paper we set $l_{\text{min}}$ to zero, which reduces \textit{P3} to the \textit{LinDist} model. $l_{\text{max}}$ is obtained by solving a power flow where the dispatchable resources are set to their capacity to determine the worst case of $l$. Based on this $l_{\text{max}}$, \textit{(P4)} is solved to obtain a value of $p^-_c$ that respects network constraints, which could be violated if \textit{LinDist} model was used instead. Based on these results, a $\Delta p_c$ which represents a feasible operating region of the network can be obtained. Simulation results are presented in the next section to show the validity of the proposed approach.
\begin{figure}[ht]
\centering
\includegraphics[width=0.4\textwidth]{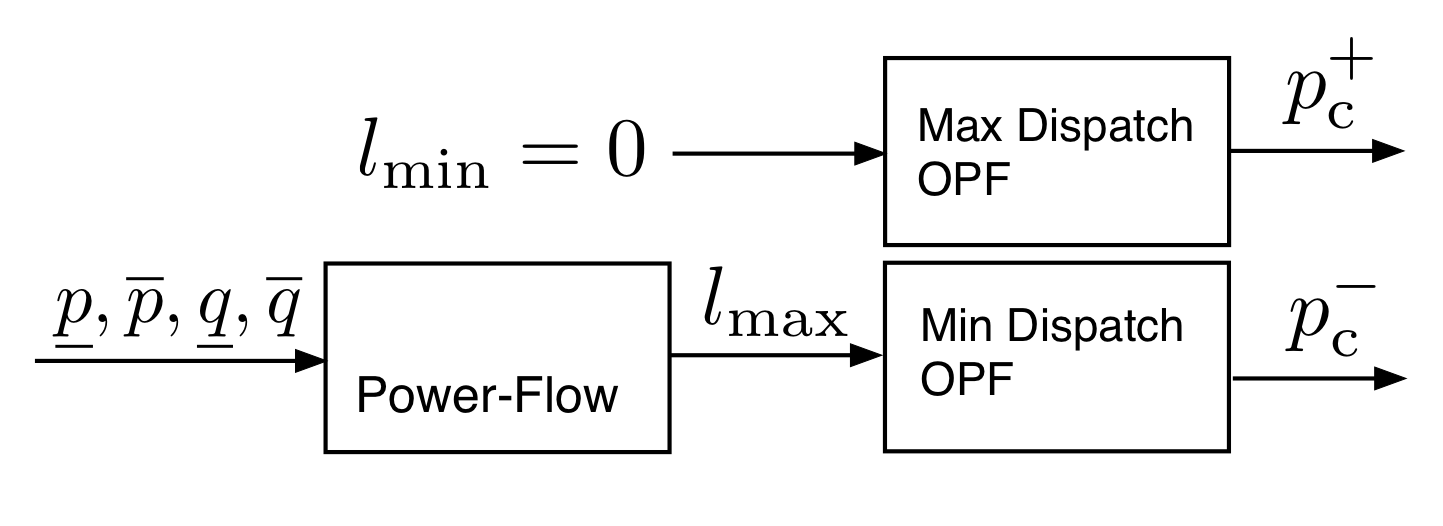}
\caption{\label{fig:block_model} Block diagram of the optimization problem to obtain the maximum dispatchable operating region.}
\end{figure}

\section{Simulation Results}\label{sec:sim_results}
In this section, simulation results on a standard IEEE-test network will be presented to show the the efficacy of the proposed approach. A comparison between the results of the convex inner approximation and the linearized OPF (\textit{LinDist}) is presented. However, the techniques presented here can be extended to improve any linearized OPF model. It is shown that the convex inner approximation results in a feasible solution when the \textit{LinDist} model may result in an infeasible solution. Simulation tests are conducted on the IEEE-13 node test case~\cite{kersting2001radial} using optimization solver SDPT3~\cite{tutuncu2003solving} in CVX, with validation of the results performed through Matpower~\cite{zimmerman2011matpower}. The block diagram in Fig.~\ref{fig:block_model} shows the steps to determine the feasible operating region from the proposed inner convex approximation.

For the purpose of comparison between the convex inner approximation and the linearized OPF model, three test scenarios are considered as shown in Table~\ref{table_scenarios}. In each of the three test scenarios, the operating regions obtained from the \textit{LinDist} OPF model and the convex inner approximation are obtained and then the feasibility of the operating regions is compared through powerflow solutions.

\begin{table}[h!]
\centering
\caption{\label{table_scenarios}Types of apparent power constraints}

{
\begin{tabular}{rll} 
\toprule
{Case no.} & {Case description} & {Bounds $(f(p_i,q_i))$} \\
\midrule
            1 & Unity power factor & $q_i=\gamma_i p_i$\\
            2 & Box constraint & $\underline{q_i}\le q_i\le \overline{q_i}$\\
            3 & Quadratic constraint & $p_i^2+q_i^2\le \overline{S_i}$
\\ \bottomrule
\end{tabular}
}
\end{table}
\begin{figure}[ht]
\centering
\includegraphics[width=0.43\textwidth]{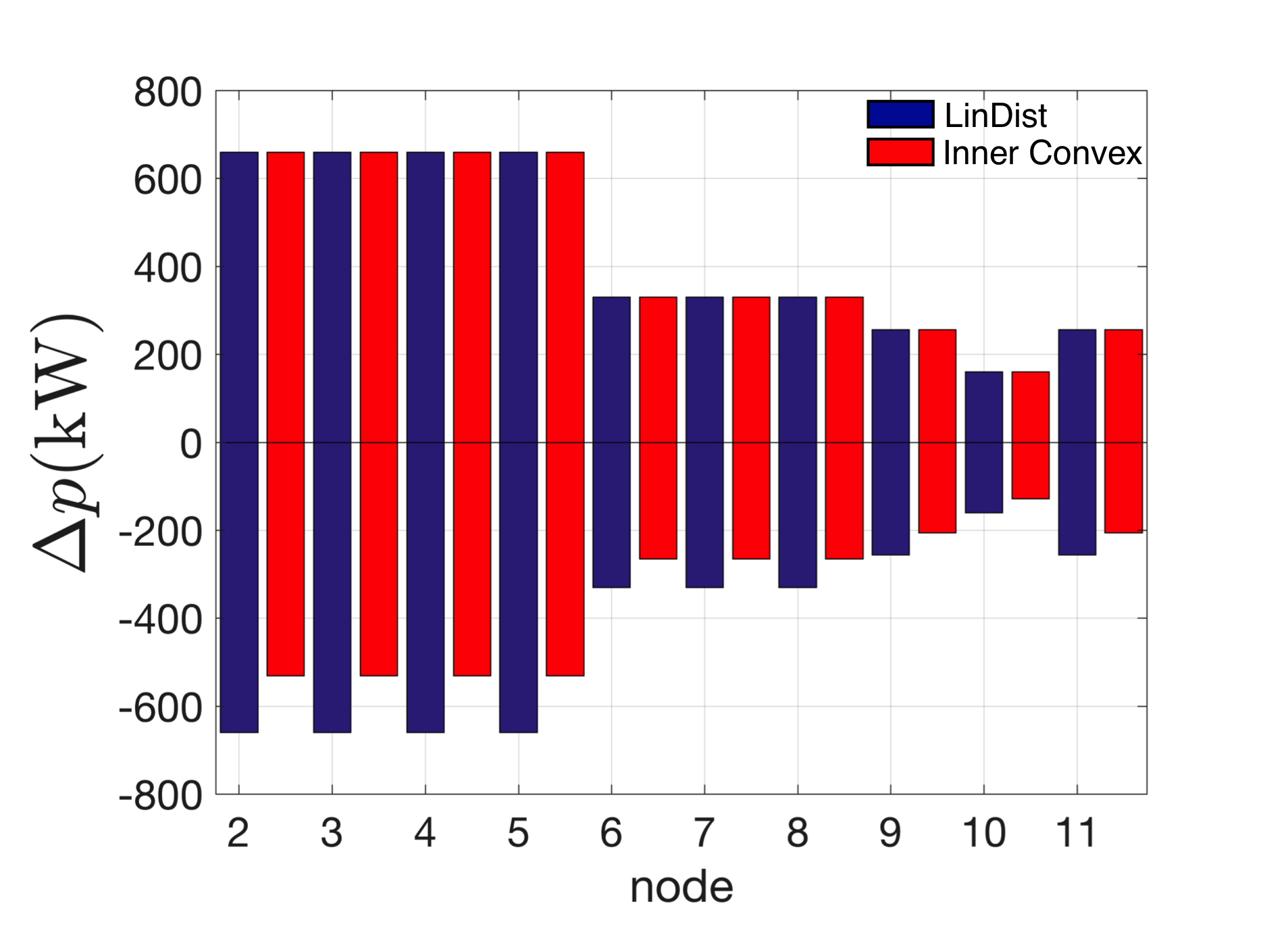}
\caption{\label{fig:delP_comp_pf} Comparison of the feasible operating regions for Case~1 (unity power factor) between \textit{LinDist} and convex inner approximation showing the conservativeness of the inner convex approximation over \textit{LinDist}. Based on this, the feasible operating region of flexible dispatchable resources can be obtained by subtracting solar and demand forecast.}
\end{figure}

\begin{figure}[t]
    \centering
    \subfloat[\label{fig:delV_comp_pos_pf}]{
    \includegraphics[width=0.51\linewidth]{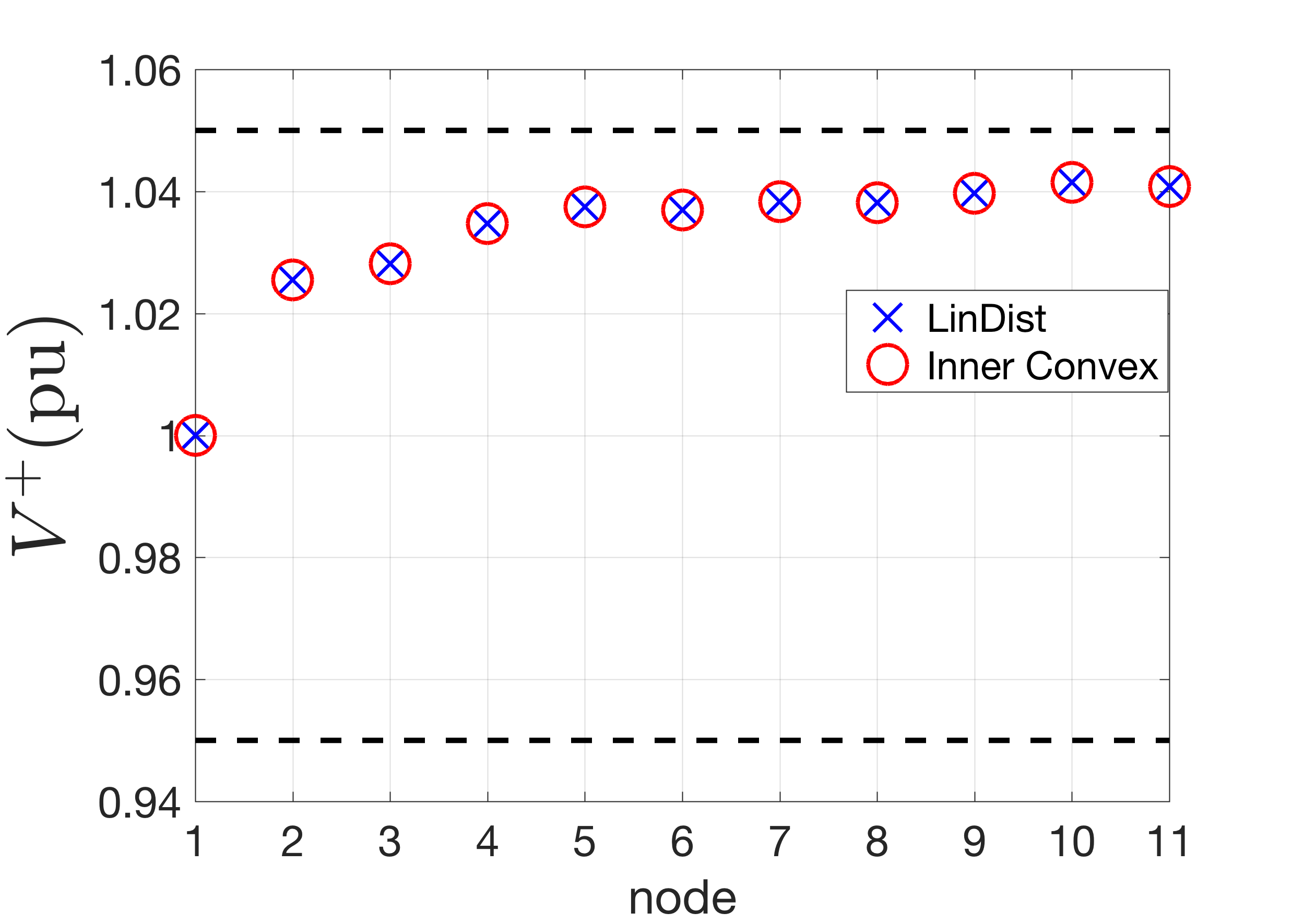}}
    \subfloat[\label{fig:delV_comp_neg_pf}]{
    \includegraphics[width=0.51\linewidth]{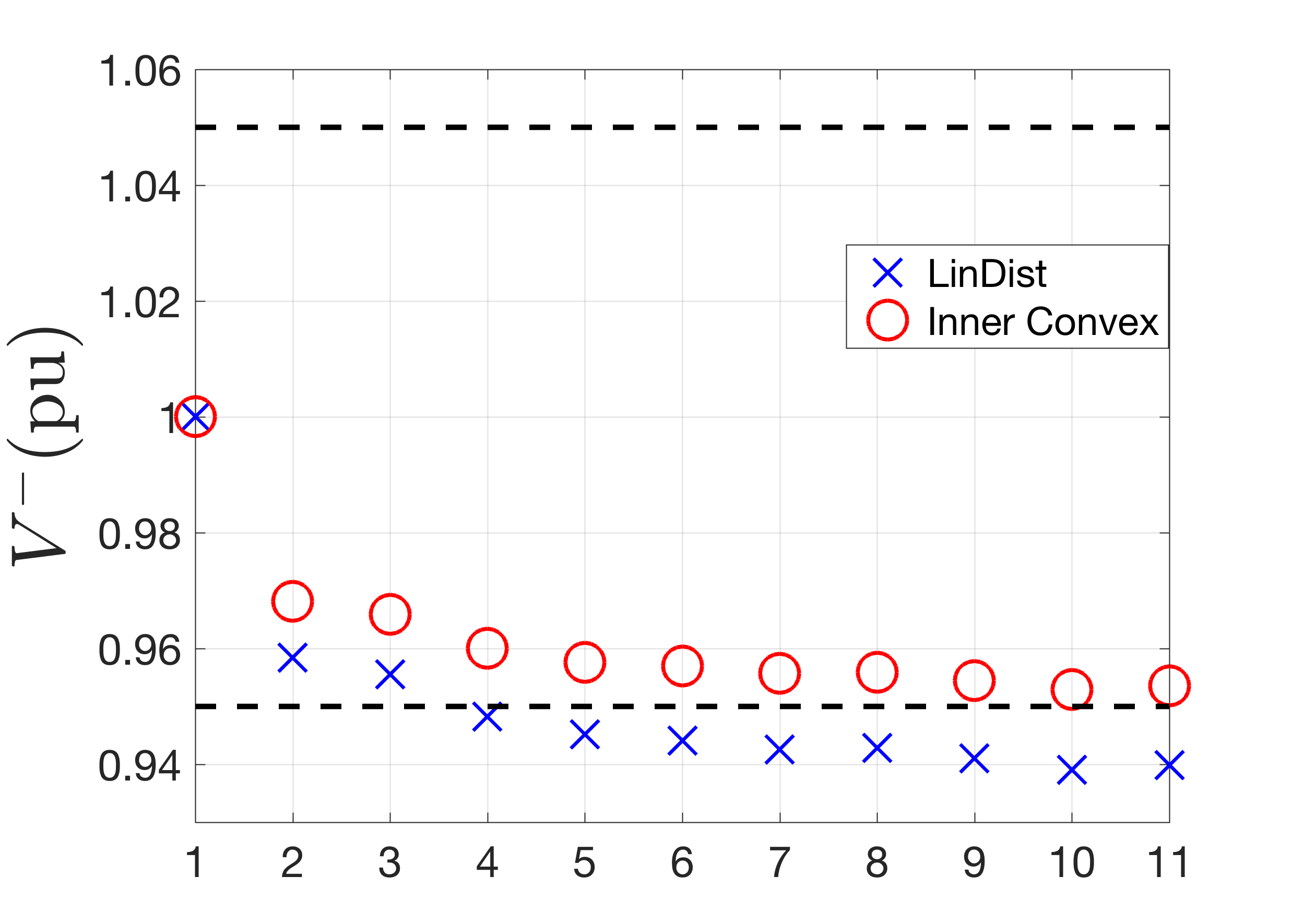}}
    \caption{(a) Comparison of the Voltages ($V^+$) for Case~1 (unity power factor) obtained through Matpower for the optimized set-points from \textit{LinDist} and convex inner approximation showing that the voltages match (as $l_{min}=0$ and are within the bounds, (b) Comparison of the Voltages ($V^-$) for Case~1 (unity power factor) obtained through Matpower for the optimized set-points from \textit{LinDist} and convex inner approximation showing that the voltages violate the bounds when using \textit{LinDist} model, which is avoided when using the convex inner approximation.}
\end{figure}

\begin{figure}[t]
    \centering
    \subfloat[\label{fig:delP_comp_box}]{
    \includegraphics[width=0.51\linewidth]{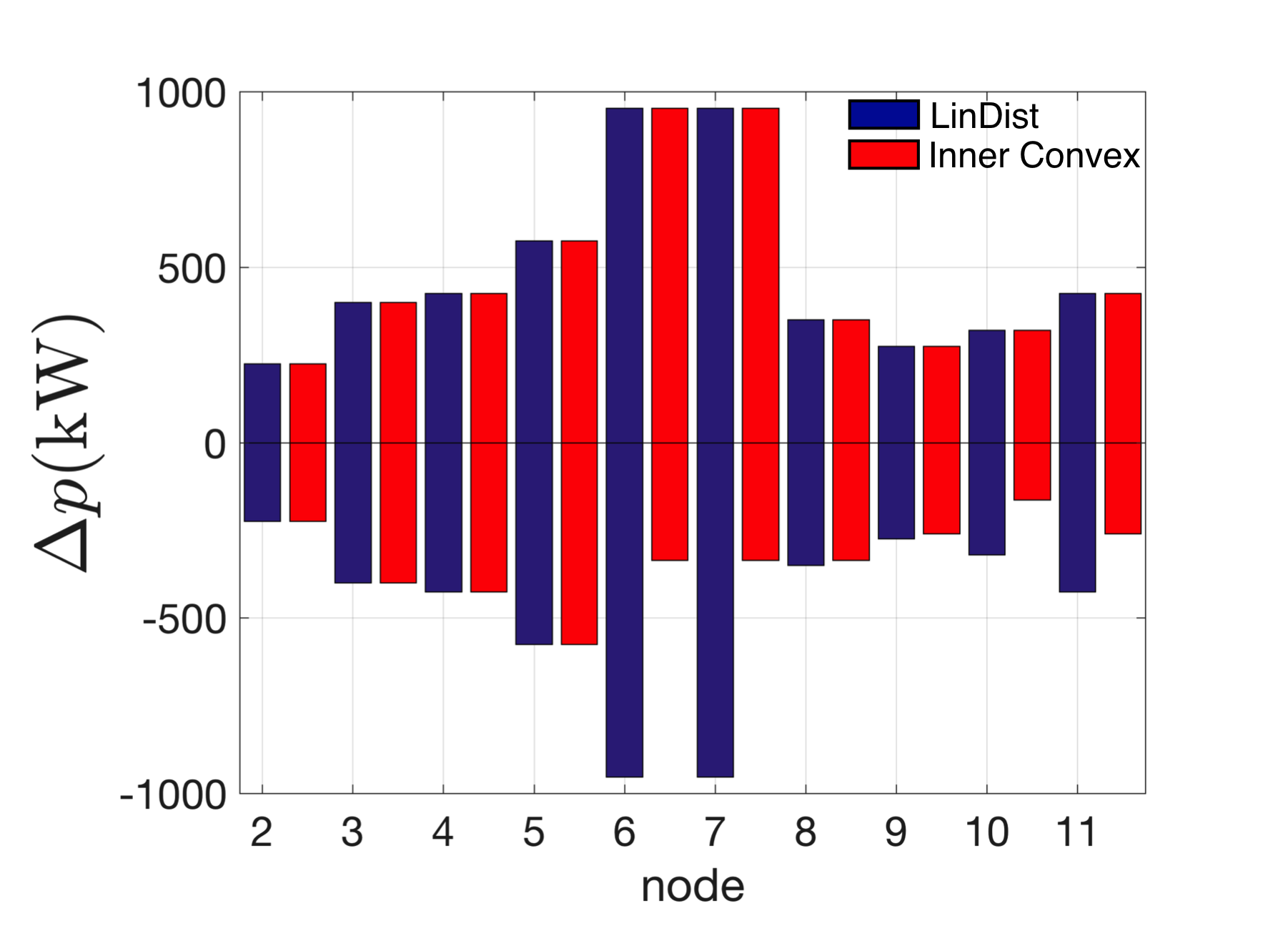}}
    \subfloat[\label{fig:delQ_comp_box}]{
    \includegraphics[width=0.51\linewidth]{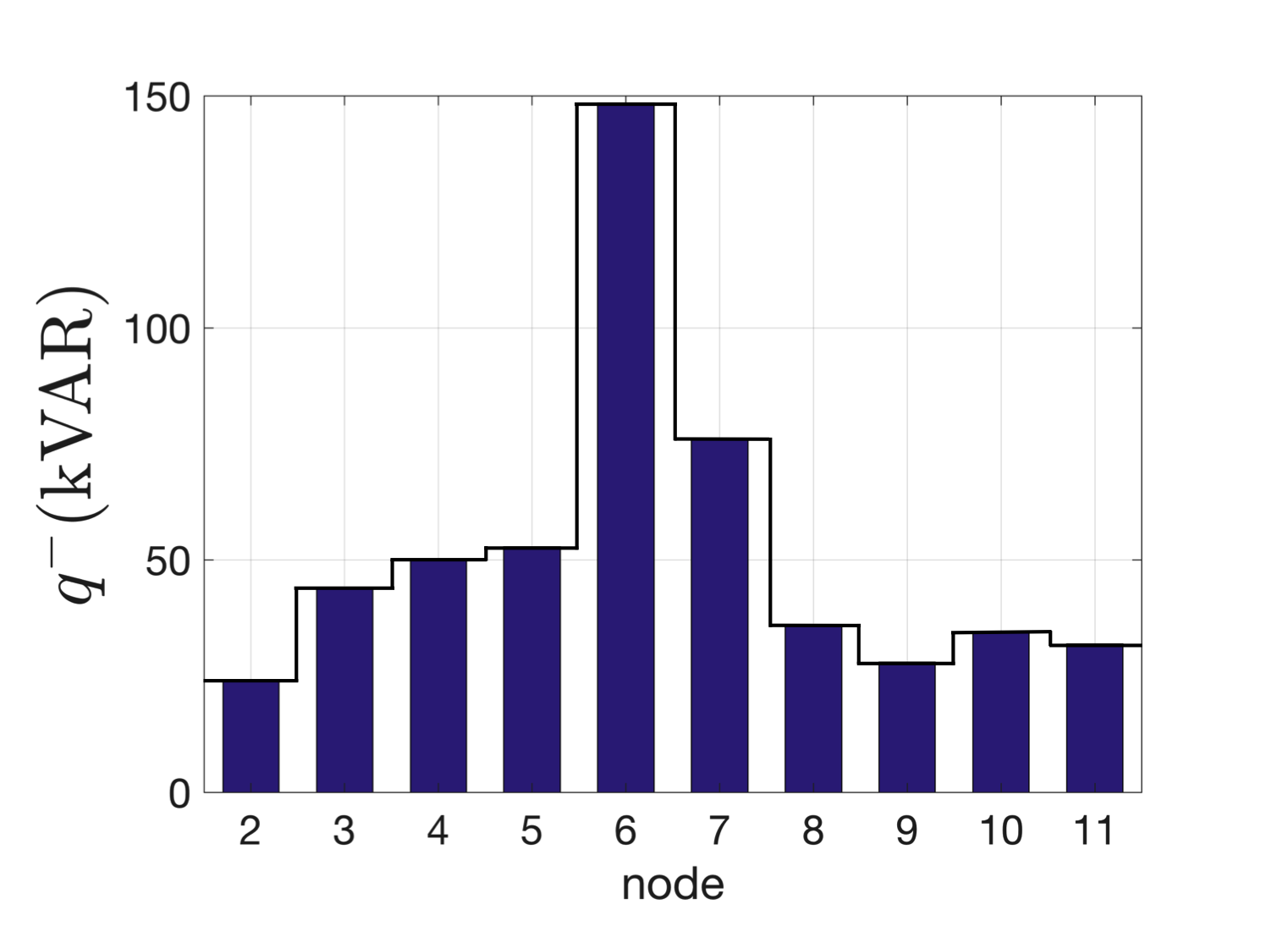}}
    \caption{(a) Comparison of the feasible operating regions for Case~2 (box constraints) between \textit{LinDist} and convex inner approximation showing the conservativeness of the inner convex approximation over \textit{LinDist}, (b) Reactive power injection (blue bars) in \textit{LinDist} for Case~2 (box constraints) showing that the reactive power is at its limit (black lines) and hence the voltage infeasibility cannot be improved through  more injection of reactive power, highlighting the need for a conservative estimate of $\Delta p$.}
\end{figure}


\begin{figure}[t]
    \centering
    \subfloat[\label{fig:delV_comp_pos_box}]{
    \includegraphics[width=0.51\linewidth]{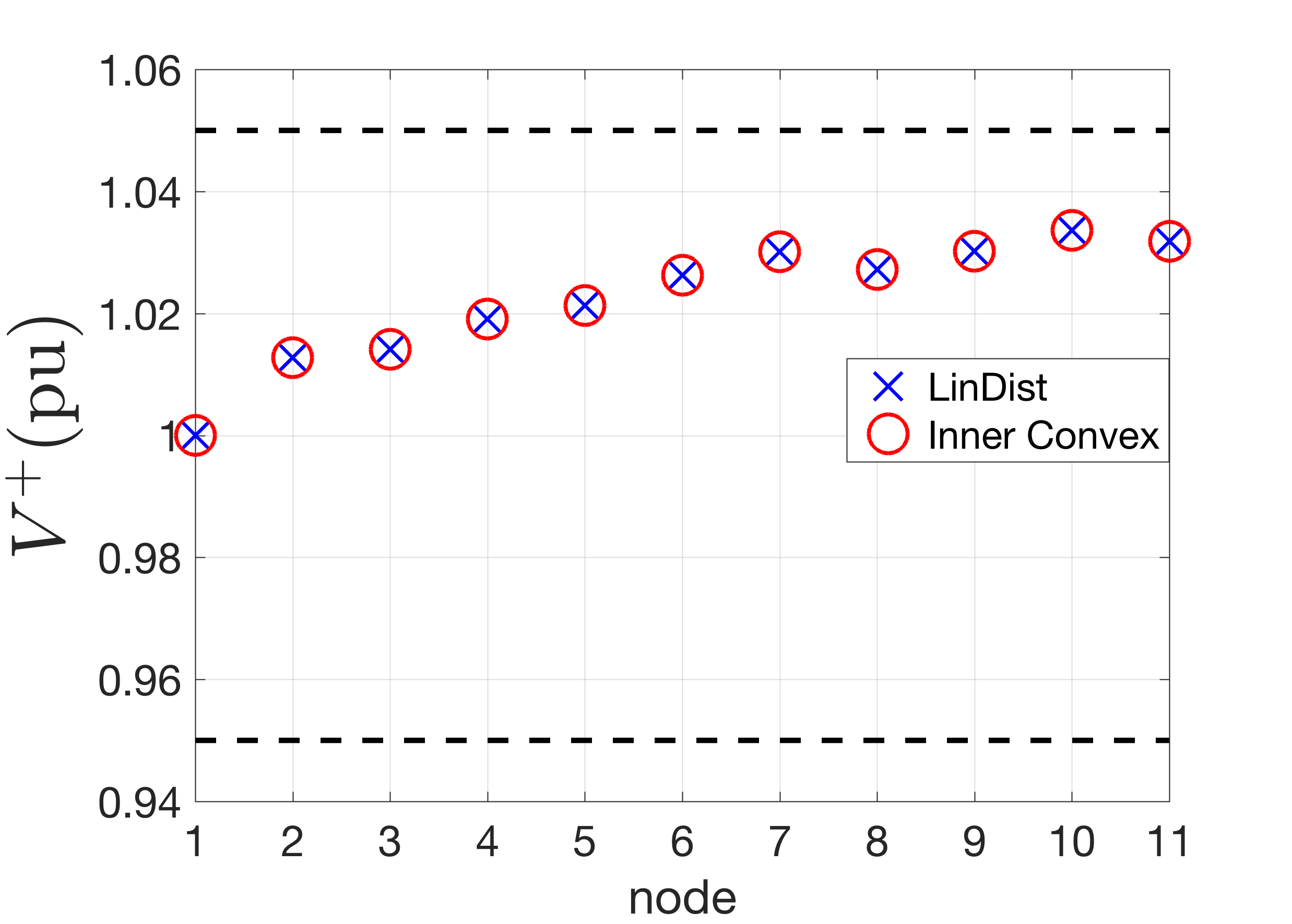}}
    \subfloat[\label{fig:delV_comp_neg_box}]{
    \includegraphics[width=0.51\linewidth]{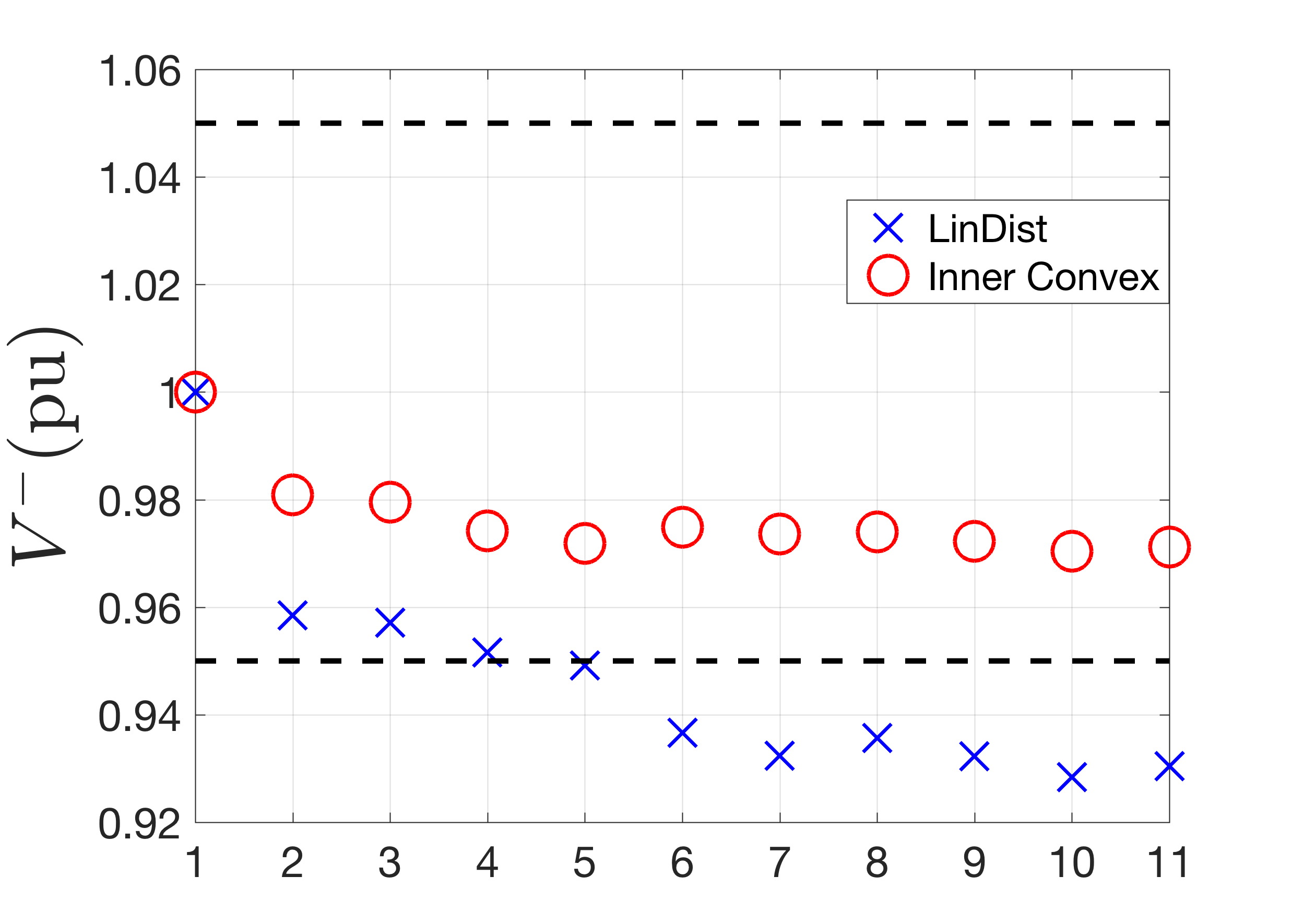}}
    \caption{(a) Comparison of the Voltages ($V^+$) for Case~2 (box constraints) obtained through Matpower for the optimized set-points from \textit{LinDist} and convex inner approximation, (b) Comparison of the Voltages ($V^-$) for Case~2 (box constraints) obtained through Matpower for the optimized set-points from \textit{LinDist} and convex inner approximation showing that the voltage bounds are violated when using the \textit{LinDist} model, which is avoided when using the convex inner approximation.}
\end{figure}


\subsection{Case 1: Constant power factor}
In this case, the relation between real and reactive power injection for the flexible resource is given by a constant power factor, $PF_i:=\cos(\phi_i)$, at node $i$:
\begin{align}\label{eq:pf_rel}
         q_i=&\gamma_i p_i \quad \forall i\in \mathcal{N}
\end{align}
where $\gamma_i:= \sqrt{(1-PF_i^2)/PF_i^2}$. For this test case, we consider the flexible resource to be an aggregation of resistive water heaters, which results in unity power factor, i.e., $cos(\phi_i)=1, \forall i\in \mathcal{N}$. 

Figure~\ref{fig:delP_comp_pf} shows the comparison of the feasible region obtained from \textit{LinDist} model and from the proposed convex inner approximation for the constant power factor case. 
It can be seen from Fig.~\ref{fig:delP_comp_pf} that the convex inner approximation provides conservative lower bounds on the feasible region. Figure~\ref{fig:delV_comp_pos_pf} and Fig.~\ref{fig:delV_comp_neg_pf} show the comparison of the voltages obtained through Matpower for the set-points shown in Fig.~\ref{fig:delP_comp_pf}. As can be clearly seen from Fig.~\ref{fig:delV_comp_neg_pf}, in case of the \textit{LinDist} model the voltages obtained when determining the lower limit of the operating region violate network bounds, whereas the convex inner approximation due to its conservative approach is able to obtain a feasible solution.
\subsection{Case 2: Box constraints on reactive power}
In this case, simulation results are conducted with  box constraints on reactive power as shown below:
\begin{align}
    \underline{q_i}\le q_i \le \overline{q_i} \quad \forall i\in \mathcal{N}
\end{align}
where $\underline{q}\in \mathbb{R}$ is the lower limit of the reactive power injection and $\overline{q} \in \mathbb{R}$ is the upper limit, at node $i$. From Fig.~\ref{fig:delP_comp_box} it can be seen that the convex inner approximation provides a conservative estimate to the feasible region. Figure~\ref{fig:delV_comp_pos_box} shows the comparison of the voltages at the upper limit of the feasible region, whereas Fig.~\ref{fig:delV_comp_neg_box} shows the comparison of the voltages at the lower limit of the feasible region. As can be seen the figures, the convex inner approximation is able to provide a feasible solution, while the \textit{LinDist} model results in violation of constraints. Figure~\ref{fig:delQ_comp_box} shows that the violation of voltage constraints in \textit{LinDist} model is due to the operating region of real power injection $\Delta p$, as the reactive power injection is at its limit and cannot provide any further voltage support. This is also proved in Theorem~\ref{Theorem1} in Appendix~\ref{appen1} that shows for the \textit{LinDist} model the reactive power at optimality is always at the boundary of the constraint.


\begin{figure}
    \centering
    \subfloat[\label{fig:delP_comp_inv}]{
    \includegraphics[width=0.51\linewidth]{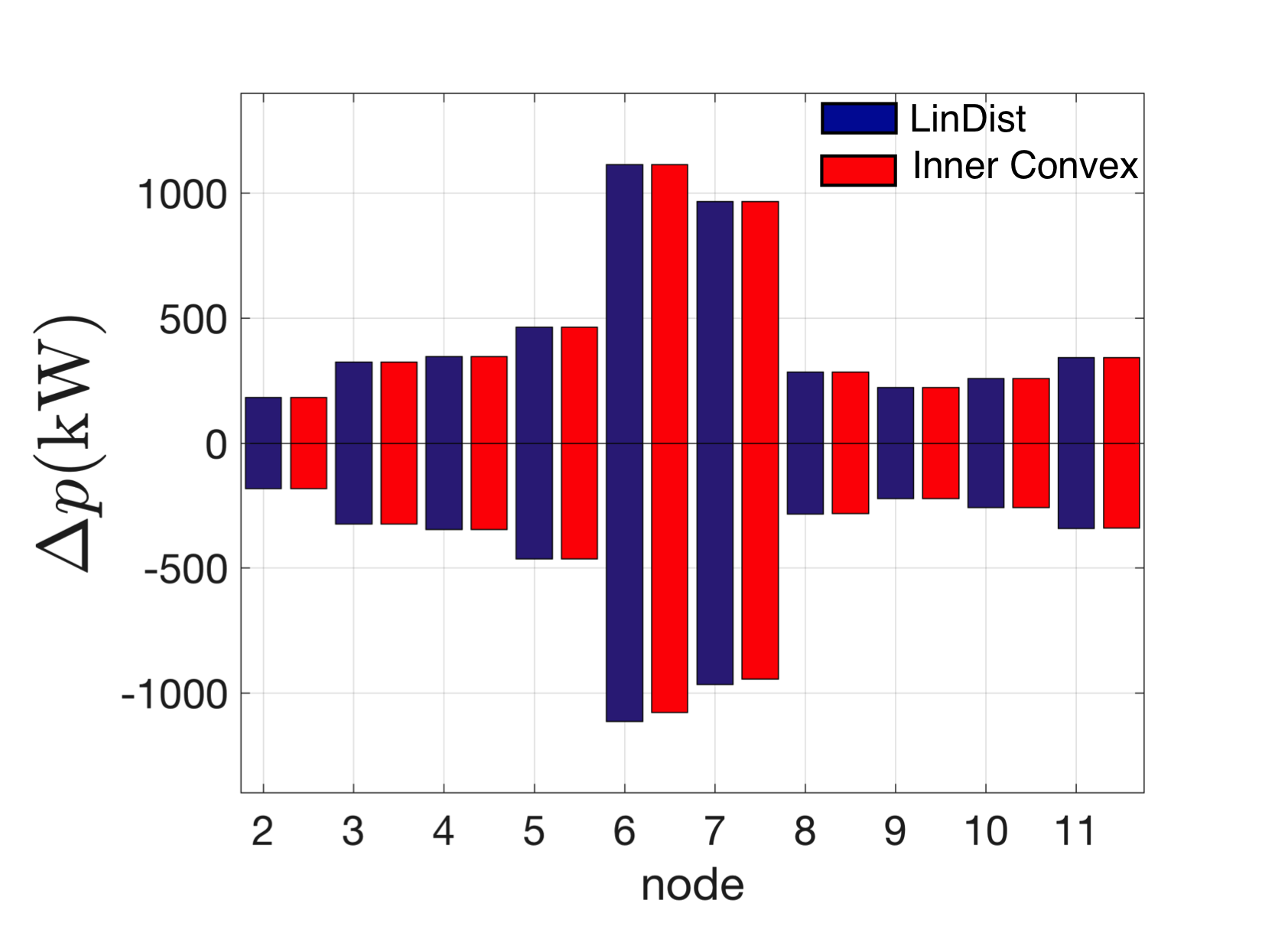}}
    \subfloat[\label{fig:Smax_comp_inv}]{
    \includegraphics[width=0.51\linewidth]{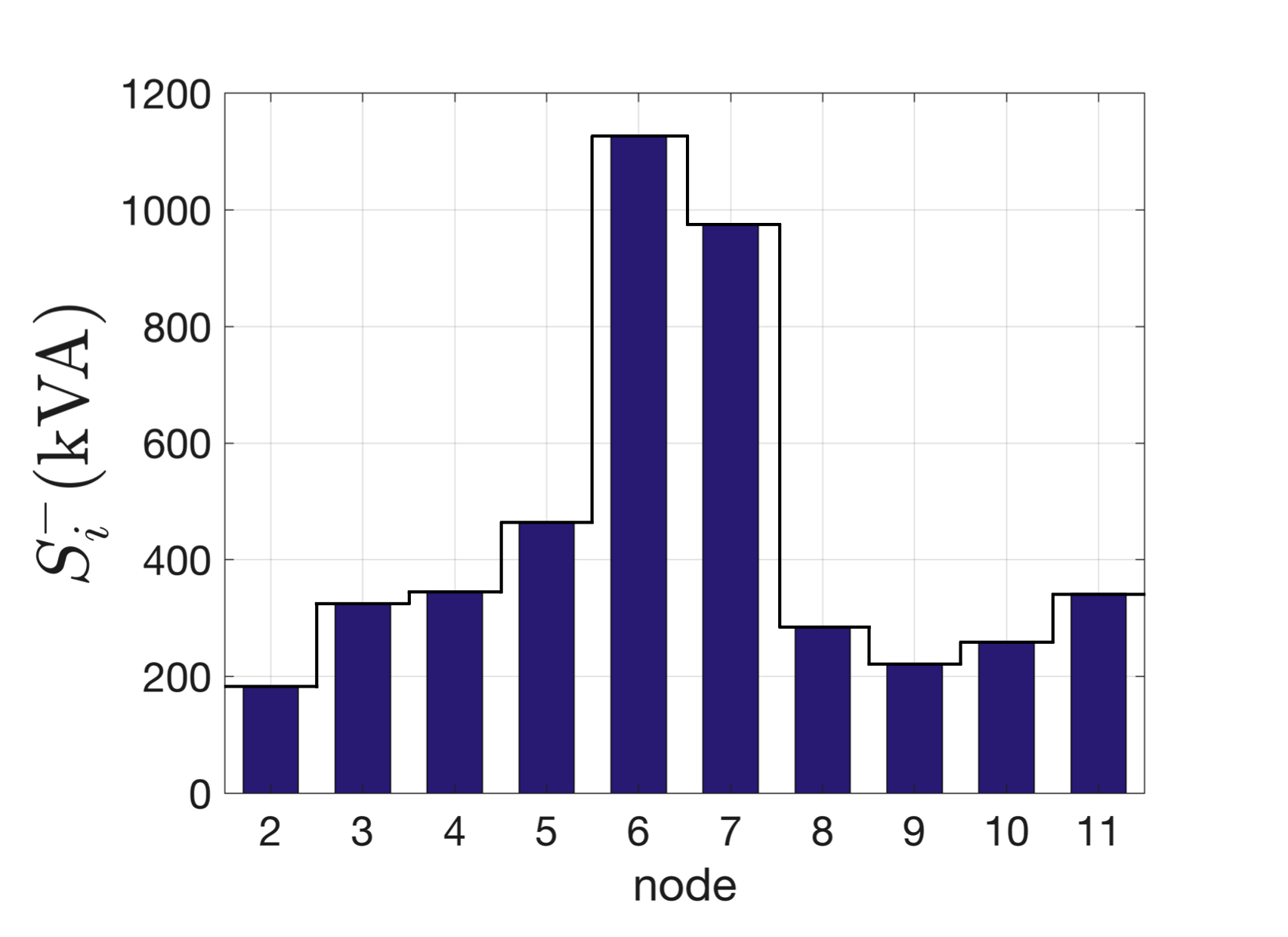}}
    \caption{(a) Comparison of the feasible operating regions for Case~3 (quadratic apparent power constraint) between \textit{LinDist} and convex inner approximation showing the conservativeness of the inner convex approximation over \textit{LinDist}, (b) Apparent power injection in \textit{LinDist} (blue bars) for Case~3 (quadratic apparent power constraint) shows that the apparent power is at its limit (black line) and hence a voltage infeasibility cannot be improved through additional injection of reactive power, highlighting the need for a conservative estimate of $\Delta p$.}
\end{figure}

\begin{figure}
    \centering
    \subfloat[\label{fig:delV_comp_pos_inv}]{
    \includegraphics[width=0.51\linewidth]{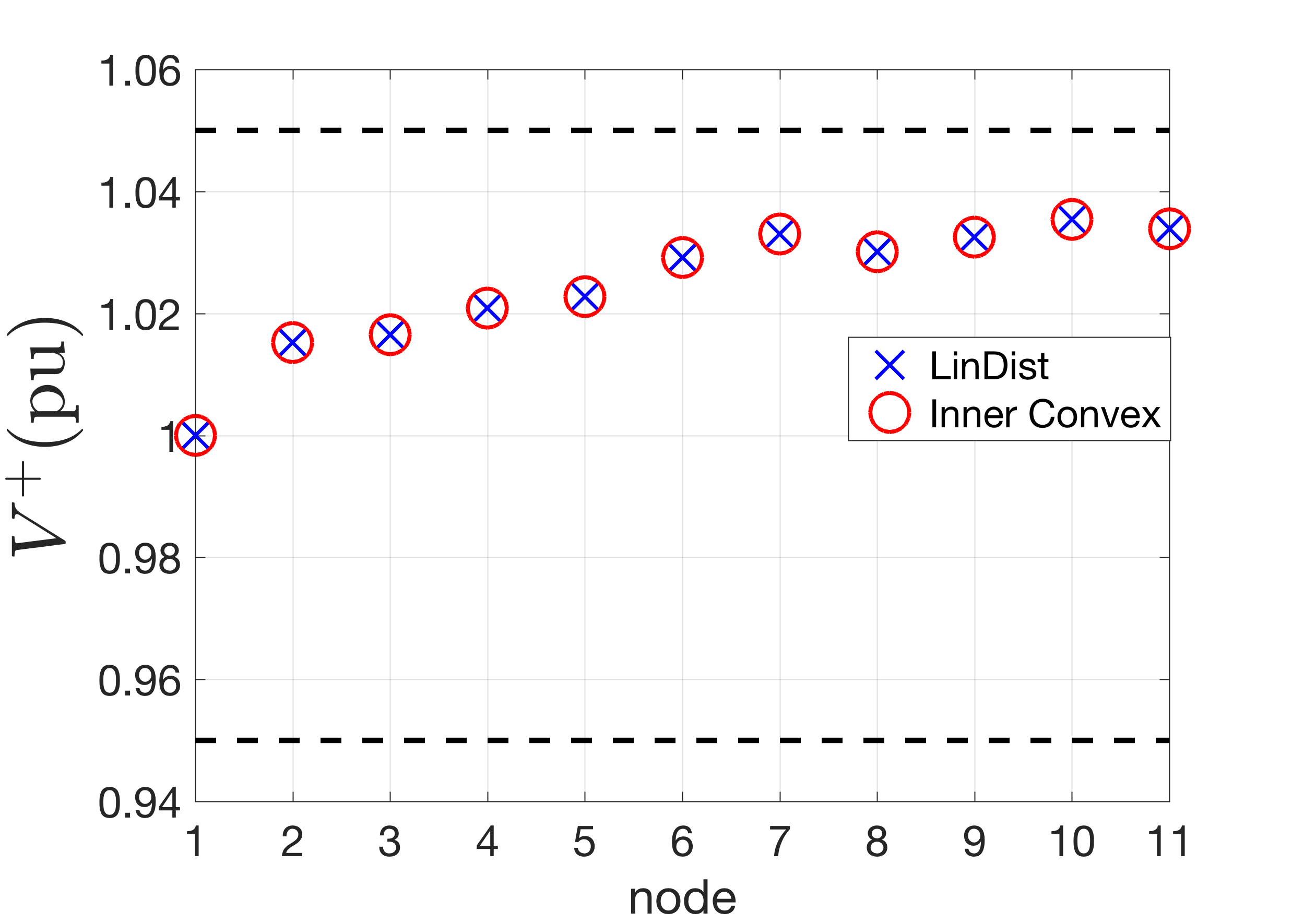}}
    \subfloat[\label{fig:delV_comp_neg_inv}]{
    \includegraphics[width=0.51\linewidth]{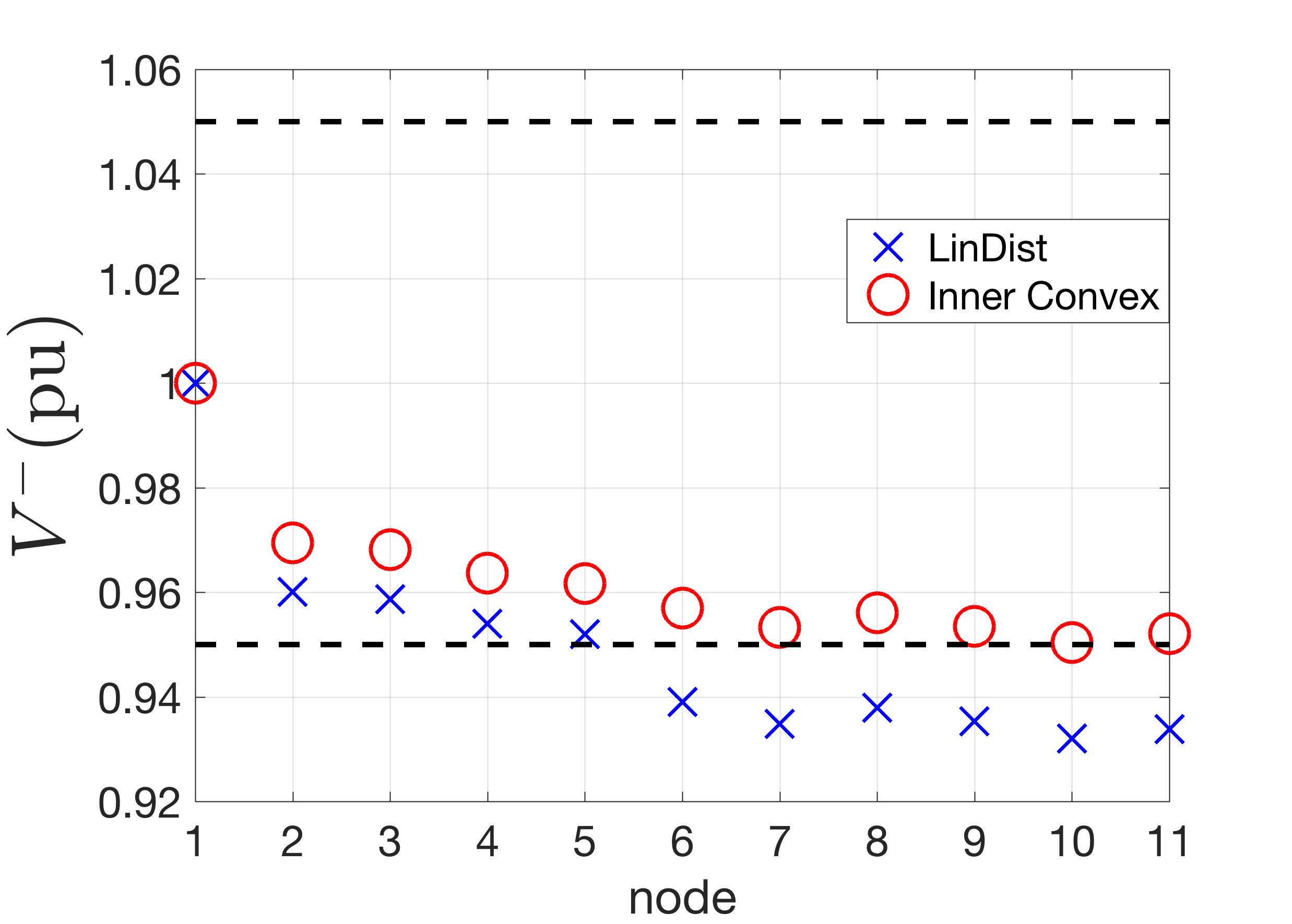}}
    \caption{(a) Comparison of the Voltages ($V^+$) for Case~3 (quadratic apparent power constraint) obtained through Matpower for the optimized set-points from \textit{LinDist} and convex inner approximation, (b) Comparison of the Voltages ($V^-$) for Case~3 (quadratic apparent power constraint) obtained through Matpower for the optimized set-points from \textit{LinDist} and convex inner approximation showing that the voltage bounds are violated when using the \textit{LinDist} model, which is avoided when using the convex inner approximation.}
\end{figure}


\subsection{Case 3: Apparent power constraint}
In this case the real and reactive power injections at each node are bound by the following quadratic apparent power constraint:
\begin{align}
    p_i^2+q_i^2\le \overline{S_i}\quad \forall i\in \mathcal{N}
\end{align}
This is the case when flexible resources are connected to the grid through inverters. From Fig.~\ref{fig:delP_comp_inv} it can be seen that the convex inner approximation provides a conservative estimate of the feasible region. The comparison of the voltages in Fig.~\ref{fig:delV_comp_pos_inv} and Fig.~\ref{fig:delV_comp_neg_inv} shows that at the lower limit of the feasible region, the \textit{LinDist} model does not result in a feasible solution, which on the other hand is provided by the convex inner approximation. Furthermore, Fig.~\ref{fig:Smax_comp_inv} shows that the nodal injections are at their apparent power limits without any margin to provide reactive power support. Theorem~\ref{Theorem1} in Appendix~\ref{appen1} presents a formal proof that shows that this condition always holds.

These simulation results illustrate how the convex inner approximation compares against the linearized OPF models with respect to modeling errors, and hence ensuring robust and resilient operation of distribution feeders. To verify the feasibility guarantee of the convex inner approximation, Monte Carlo methods are employed to sample 10,000 different combinations of the DER operating regions and simulate the resulting AC load flow with Matpower to determine nodal voltages with distribution as shown in Fig.~\ref{fig:volt_spread}. From the figure, it is clear that voltages are within limits for all combinations. 

\begin{figure}[ht]
\centering
\includegraphics[width=0.4\textwidth]{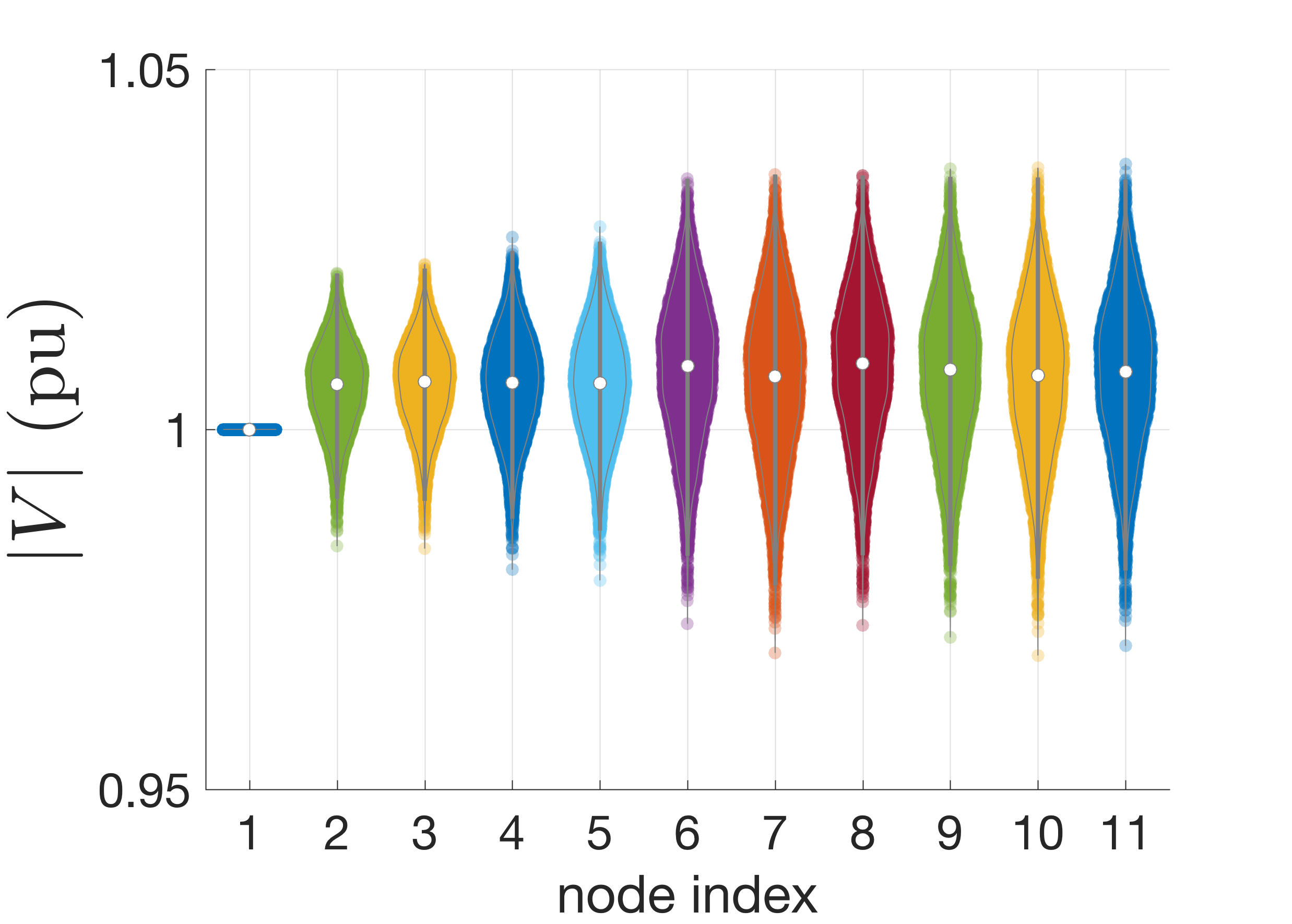}
\caption{\label{fig:volt_spread} Violin plot of the voltage magnitudes from 10,000 Monte-Carlo simulations using the operation regions obtained from the convex inner approximation.}
\end{figure}

\section{Conclusions and future work}\label{sec:conclusion}
This paper presents a convex inner approximation of the optimal power flow problem that determines the feasible operating region of the dispatchable resources while respecting the network constraints. Through illustrative examples, the shortcomings of linear OPF techniques is explained. The mathematical formulation of the convex inner approximation is developed and through simulations on an IEEE test network, the advantage of using this formulation is presented. 

Future work will extend this technique to other OPF formulations and to multi-period problems in the OPF domain to consider the challenging case of a reference trajectory that needs to be disaggregated over time across energy-constrained nodes on a network. Extending this work further to include stochastic distributed generation and demand effects while providing feasibility guarantees under both model mismatch and forecast uncertainty is an important area of work to be undertaken. Finally, while we focus on a feasible disaggregation at the nodal level (to ensure constraints are not violated), we should also need to consider an optimal disaggregation policy that allows an aggregator to steer the system response along an (online) OPF solution.

\appendices
\section{LinDist model at optimality}\label{appen1}
Based on the mathematical modeling presented in section~\ref{sec:math_model}, the optimization problem to determine the operating region of dispatchable resources using \textit{LinDist} model can be expressed as:
\begin{subequations}\label{eq:P5}
\begin{align}\label{eq:P5_a}
\text{(P5)} \quad &  \max_{V, p, q} \ \sum_{i=1}^{n}\log( p_i)\\
\text{subject to}: \quad & M_{\text{p}}p = V-v_0\mathbf{1}-M_{\text{q}}q \label{eq:P5_b}\\
&   \underline{S}_i \le f(p_i,q_i) \le \overline{S}_i \quad \forall i\in \mathcal{N}\label{eq:P5_c}\\
& \underline{V}\leq V\leq \overline{V}\label{eq:P5_d}
\end{align}
\end{subequations}

\begin{theorem}\label{Theorem1}
With the \textit{LinDist} model applied to a radial, inductive distribution feeder, the  constraints related to reactive power injections from the dispatchable demand-side resources are all active (i.e., $f(p_i,q_i) = \underline{S}_i$ or $f(p_i,q_i) = \overline{S}_i$ for all $i$).
\end{theorem}
\begin{proof}(by contradiction): Assume that at optimality the reactive power is not at its constraint (i.e., $\exists \, i, \text{s.t} \ \underline{S_i} <  f(p_i,q_i) < \overline{S_i} $). Then, the Lagrange multiplier associated with~\eqref{eq:P5_c} is zero and, hence, this constraint will not show up in the KKT conditions.

The voltage constraint in~\eqref{eq:P5_d} can be expressed as:
\begin{align}\label{eq:volt_const_th}
    \underline{V}\le M_{\text{p}}p+M_{\text{q}}q+v_{\text{0}}\mathbf{1}\le \overline{V}
\end{align}
Let $\underline{\lambda} \in \mathbb{R}^n$ and $\overline{\lambda} \in \mathbb{R}^n$ be the Lagrange multipliers associated with \eqref{eq:P5_d} and let $\mathcal{L}$ be the Lagrangian. From the KKT conditions, the following relations are obtained:
\begin{align}
    \frac{\partial \mathcal{L}}{\partial p_i}=&\frac{1}{p_i}-\underline{\lambda}^i[M_{\text{p}}]_i+\overline{\lambda}^i[M_{\text{p}}]_i=0\label{eq:KKT1}\\
     \frac{\partial \mathcal{L}}{\partial q_i}=&-\underline{\lambda}^i[M_{\text{q}}]_i+\overline{\lambda}^i[M_{\text{q}}]_i=0\label{eq:KKT2}
\end{align}
where $[M_{\text{p}}]_i$ and $[M_{\text{q}}]_i$ are the sums of the $i$th columns of $M_{\text{p}}$ and $M_{\text{q}}$, respectively. From \eqref{eq:KKT2}, since $[M_{\text{q}}]_i>0$ for inductive networks, we have that $\alpha_i\overline{\lambda}^i-\beta_i\underline{\lambda}^i=0$ for some $\alpha_i, \beta_i>0$, which implies that $\overline{\lambda}^i=\underline{\lambda}^i=0$. Substituting this result in \eqref{eq:KKT1}, gives $\frac{1}{p_i}=0$, which is not possible for $p_i$ finite. Thus, we reach a contradiction. Hence, constraint~\eqref{eq:P5_c} must be active. In other words, the reactive power at optimality is always at the boundary of its constraint.
\end{proof}
\bibliographystyle{IEEEtran}
\small\bibliography{fix.bib}
\end{document}